\documentclass[a4paper,reqno]{amsart}

\usepackage[T1]{fontenc}
\usepackage[utf8]{inputenc}
\usepackage{lmodern}
\usepackage[american]{babel}
\usepackage{microtype}
\usepackage{fullpage}
\usepackage[hidelinks,hypertexnames=false]{hyperref}
\usepackage{enumitem}
\usepackage{csquotes}
\usepackage{upgreek}
\usepackage{amsmath,amssymb}
\usepackage{amsthm,thmtools}
\usepackage{mathtools}
\usepackage{stmaryrd}
\usepackage{ifthen}
\usepackage{bbm}

\usepackage{orcidlink}

\usepackage[nosort,nocompress]{cite}
\usepackage[nameinlink]{cleveref}

\DeclareMathOperator{\dom}{dom}

\DeclareMathOperator{\cls}{\overline{span}}

\DeclareMathOperator{\ran}{ran}
\DeclareMathOperator{\cran}{\overline{ran}}
\DeclareMathOperator{\spt}{spt}
\DeclareMathOperator{\Real}{Re}

\DeclareMathOperator{\diver}{div}
\DeclareMathOperator{\diverz}{div_0}
\DeclareMathOperator{\grad}{grad}
\DeclareMathOperator{\gradz}{grad_0}
\DeclareMathOperator{\curl}{curl}
\DeclareMathOperator{\curlz}{curl_0}
\DeclareMathOperator{\abscb}{s_b}

\newcommand{\euler}{\mathrm{e}}
\newcommand{\iu}{\mathrm{i}}
\newcommand{\multm}{\mathrm{m}}
\newcommand{\sH}{\mathrm{H}}

\newcommand{\R}{\mathbb{R}}
\newcommand{\N}{\mathbb{N}}
\newcommand{\C}{\mathbb{C}}
\newcommand{\1}{\mathbf{1}}

\newcommand{\FF}{\mathcal{F}}
\newcommand{\LL}{\mathcal{L}}
\newcommand{\TT}{\mathcal{T}}
\newcommand{\HH}{\mathcal{H}}
\newcommand{\Lp}[2][]{\mathrm{L}^{#2}\ifthenelse{\equal{#1}{}}{}{_{#1}}} 
\newcommand{\Lb}{\mathrm{L}_{\mathrm{b}}} 
\newcommand{\Cc}[1][\infty]{\mathrm{C}_{\mathrm{c}}^{#1}}
\newcommand{\integral}[4]{\int_{#1}^{#2} #3 \operatorname{d}\! #4}

\newcommand{\abs}[1]{\lvert#1\rvert}
\newcommand{\bigabs}[1]{\Big\lvert#1\Big\rvert}
\newcommand{\norm}[2][]{\lVert#2\rVert\ifthenelse{\equal{#1}{}}{}{_{#1}}}
\newcommand{\bignorm}[2][]{\Big\lVert#2\Big\rVert\ifthenelse{\equal{#1}{}}{}{_{#1}}}
\newcommand{\iprod}[3][]{\langle#2,#3\rangle\ifthenelse{\equal{#1}{}}{}{_{#1}}}
\newcommand{\bigiprod}[3][]{\Big\langle#2,#3\Big\rangle\ifthenelse{\equal{#1}{}}{}{_{#1}}}
\newcommand{\dprod}[3][]{[#2,#3]\ifthenelse{\equal{#1}{}}{}{_{#1}}}
\newcommand{\from}{\colon}

\numberwithin{equation}{section}

\declaretheorem[name=Definition,style=definition,numberwithin=section,qed=$\ast\ast$]{definition}
\declaretheorem[name=Remark,style=remark,numberlike=definition,qed=$\sslash$]{remark}
\declaretheorem[name=Example,style=remark,numberlike=definition,qed=$\sslash$]{example}
\declaretheorem[name=Lemma,numberlike=definition]{lemma}
\declaretheorem[name=Corollary,numberlike=definition]{corollary}

\declaretheorem[name=Theorem,numberlike=definition]{theorem}

\declaretheorem[name=Picard's Well-Posedness Theorem,numberlike=definition]{thmPicard}
\declaretheorem[name=Open Problem,numberlike=definition]{openproblem}

\title[Duality for Ev. Eqns.]{Duality for Evolutionary Equations with Applications to Null Controllability}
\date{\today}

\author[A.~Buchinger]{Andreas Buchinger\,\orcidlink{0009-0004-4203-5874}}

\address[A.B.]{TU Bergakademie Freiberg \\
  Institute of Applied Analysis \\
  Akademiestr.\ 6 \\
  D-09596 Freiberg \\
  Germany}
\curraddr{Technische Universit\"at Hamburg \\
  Institut f\"ur Mathematik \\
  Am Schwarzenberg-Campus 3 \\
  D-21073 Hamburg \\
  Germany}
\email{andreas.buchinger@tuhh.de}
\author[C.~Seifert]{Christian Seifert\,\orcidlink{0000-0001-9182-8687}}
\address[C.S.]{Technische Universit\"at Hamburg \\
  Institut f\"ur Mathematik \\
  Am Schwarzenberg-Campus 3 \\
  D-21073 Hamburg \\
  Germany}
\email{christian.seifert@tuhh.de}

\begin{document}

\begin{abstract}
    We study evolutionary equations in exponentially weighted $\Lp{2}$-spaces as introduced by Picard in 2009. First, for a given evolutionary equation, we explicitly describe the $\nu$-adjoint system, which turns out to describe a system backwards in time. We prove well-posedness for the $\nu$-adjoint system. We then apply the obtained duality to introduce and study notions of null-controllability for evolutionary equations.
    


\medskip
\noindent \textbf{Declaratory statements.} The authors have no conflict of interest to declare that
is relevant to the contents of this article. Data sharing is not applicable to this
article as no datasets were generated or analysed in this study.
\end{abstract}
\maketitle
\section{Introduction}
Evolutionary equations provide a model class for studying various partial differential equations in a unified way. Since their introduction by Rainer Picard in 2009 \cite{Picard2009} (see also \cite{PiMc11})
 quite a few particular situations have been considered in this framework. Let us provide a (by no means exhaustive) list of examples.
 Many classical equations such as, e.g., the heat equation, the wave equation, Maxwell's equations (see \Cref{exEvoEq}), and poro- (e.g., \cite{McPic10}) as well as
 fractional (e.g., \cite{PicTroWau15}) elasticity can be modeled as evolutionary equations. A more intricate example are certain piezoelectric coupling models. In \cite{Picard17},
 they were treated with boundary dynamics using abstract boundary data spaces
 that do not require any regularity assumptions on the boundary (see \cite{PicTrosWau2016,PicSeiTroWau16}). Moreover, in \cite{MPTW16} thermal coupling without boundary dynamics was included, and recently in \cite{BuDo24}, 
 thermo-piezoelectric coupling with boundary dynamics was modeled in this framework.
 Another example are (infinite-dimensional) differential-algebraic equations that were also treated in the evolutionary context (e.g., \cite{TroWau2019}).
 A general overview on evolutionary equations can be found in the monographs \cite{PicardMcGheeTrostorffWaurick2020, SeTrWa22}.

In order to explain evolutionary equations in a nutshell, let $\HH$ be a Hilbert space of functions of a spatial variable; we may think of $\HH\coloneqq \Lp{2}(\Omega)$ for some subset $\Omega\subseteq \R^d$. Given a (spatial) differential operator $A$ and a material law function $M(\cdot)$, i.e., a holomorphic mapping from a subset of $\C$ to the bounded operators on $\HH$, describing the constituent relations, we consider exponentially weighted Bochner--Lebesgue spaces $\Lp[\nu]{2}(\R,\HH)$, where $\nu$ is the weight parameter. Elements of $\Lp[\nu]{2}(\R,\HH)$ are interpreted as functions of time $\R$ mapping to the spatial function space $\HH$. Introducing the temporal derivative as an invertible operator $\partial_{t,\nu}$ in $\Lp[\nu]{2}(\R,\HH)$, the prototype of an evolutionary equation is of the form
\[(\partial_{t,\nu} M(\partial_{t,\nu}) + A)u = f,\]
where $M(\partial_{t,\nu})$ is given by functional calculus via the Fourier--Laplace transformation, $A$ is lifted to the Bochner space as a multiplication operator, and $f\in \Lp[\nu]{2}(\R,\HH)$.
Thus, in contrast to the prototypical abstract Cauchy problem from semigroup theory, studying evolutionary equations is based on studying inhomogeneous equations rather than initial value problems. The solution theory for evolutionary equations provides well-posedness under mild assumptions on $A$ and $M(\cdot)$; see, e.g., \cite[Chapter 6]{SeTrWa22}.

The first question we raise and answer is about duality: given an evolutionary equation, how does the adjoint equation look like (if it exists)? In order to do this, we make use of the unweighted $\Lp{2}$-space $\Lp{2}(\R,\HH)$ as pivot space; thus, the $\nu$-adjoint equation is formulated in $\Lp[-\nu]{2}(\R,\HH)$ rather than in $\Lp[\nu]{2}(\R,\HH)$. This idea already came up in~\cite{KPSTW14}.
It turns out that the $\nu$-adjoint equation is again close to an evolutionary equation; however, it is running backwards in time which can be nicely investigated by the time-reversal operator.

As a second goal of the article, we want to study control theory for evolutionary equations. Here, we want to focus on the notion of null-controllability for evolutionary equations, which to the best of our knowledge has not been studied before. Recall that null-controllability for an inhomogeneous initial value problem with a solution existing pointwise in time means that, given a final time horizon $T$, for every initial condition we can find a suitable control function essentially acting as inhomogeneity such that the solution at time $T$ is zero. We refer to \cite{Lions1988, FursikovImanuvilov1996, TucsnakWeiss2009} for null-controllability for evolution equations. When trying to put a corresponding definition for evolutionary equations, we have to deal with the difficulty that the solutions of evolutionary equations are only in $\Lp[\nu]{2}(\R,\HH)$, i.e., the time-regularity is only $\Lp{2}$ and we thus cannot make use of a pointwise statement on the solution at $T$. To overcome this issue, we propose two different notions; one tailored to the $\Lp{2}$-setting and another one which can make use of pointwise statements in a weaker spatial space. In order to study null-controllability, one key method is to establish a so-called observability estimate for the dual problem, see, e.g., \cite{Douglas-66,DoleckiR-77,Lions1988,Carja1988,LebeauR-95,Zuazua-07,TenenbaumT-11,GallaunST-20,BombachGST-23}. Since we have established the duality for evolutionary equations in the first part of the article, we can now apply this duality to prove the corresponding duality statement in the context of null-controllability for evolutionary equations. We again find the well-known equivalence between null-controllability and (a version of) observability estimates for the $\nu$-adjoint system; however due to the lack of pointwise statements, the observability estimate has to be interpreted accordingly. In a special case, we also formulate a version of a pointwise interpretation of the null-controllability and pose an open problem in this situation.

Let us outline the remaining parts of the paper. In Section \ref{sec:Preliminaries}, we will recall the basic theory of evolutionary equations, as well as the duality of null-controllability and obervability estimates in the context of abstract Cauchy problems.
Section \ref{secDefPropNuProd} is devoted to the duality of evolutionary equations. Finally, in Section \ref{sec:ControlTheory}, we first introduce a notion of null-controllability for general evolutionary equations and apply the general duality results of Section \ref{secDefPropNuProd} to establish the duality between null-controllability and an observability estimate for the adjoint evolutionary equation. We further suggest a pointwise interpretation of null-controllability.
\section{Preliminaries}
\label{sec:Preliminaries}
Hilbert spaces $\HH$ considered in this paper are always endowed with an inner product $\iprod[\HH]{\cdot}{\cdot}$ that is antilinear in the first and linear in the second argument. If $A\colon\dom(A)\subseteq \HH\to \HH$ is a closed operator on a Hilbert space $\HH$
and $\dom(A)$ is not explicitly considered as a subspace, then we consider $\dom(A)$ as the Hilbert space endowed with the inner product $\iprod[\HH]{\cdot}{\cdot}+\iprod[\HH]{A\cdot}{A\cdot}$.
\subsection{Evolutionary Equations}
 Here, we will revisit the theory of evolutionary equations that can be found in~\cite{SeTrWa22}.
 For the sake of simplicity, we will restrict ourselves to the autonomous case.

For a Hilbert space $\HH$ and a $\nu\in\R$, we consider the exponentially $\nu$-weighted $\Lp{2}$-space of $\HH$-valued (equivalence classes of) functions
\begin{equation*}
\Lp[\nu]{2} (\R,\HH)\coloneqq\big\{f \mid f\colon\R\to \HH \text{ Bochner-meas.\ and } (t\mapsto \euler^{-\nu t}\norm[\HH]{f(t)})\in\Lp{2} (\R)\big\}\text{.}
\end{equation*}

On this Hilbert space, we consider the unitary multiplication operator
\begin{equation*}
\exp(-\nu\multm)\colon
\begin{cases}
\hfill \Lp[\nu]{2} (\R,\HH) &\to\quad \Lp{2} (\R,\HH)\\
\hfill f&\mapsto\quad \euler^{-\nu\cdot}f(\cdot)
\end{cases}\text{.}
\end{equation*}
Together with the unitary Fourier transformation $\FF$ on $\Lp{2} (\R,\HH)$, this yields the unitary \emph{Fourier--Laplace transformation}
$\LL_\nu\coloneqq\FF\exp(-\nu\multm)\colon \Lp[\nu]{2} (\R,\HH) \to\Lp{2} (\R,\HH)$. The core idea now is to think of $\HH$ as the space of spatial functions
and to write the system we want to model (respectively its solution operator) as the Fourier--Laplace transformation of a holomorphic function that in each point
linearly maps $\HH$ to itself.

For $f,g\in\Lp[\nu]{2} (\R,\HH)$, we define the \emph{weak time derivative} $\partial_{t,\nu}$ in the usual way
\begin{equation*}
f\in \dom(\partial_{t,\nu})\text{ and } g=\partial_{t,\nu}f \quad :\!\!\iff\quad \forall \varphi\in\Cc (\R):\integral{\R}{}{\varphi(t) g(t)}{t}=-\integral{\R}{}{\varphi^\prime(t) f(t)}{t}\text{,}
\end{equation*}
and get (cf.~\cite[Chapter~3.2 and Proposition~4.1.1]{SeTrWa22}):
\begin{lemma}\label{lemmaWeakTimeDerProperties}
For $\nu\in\R$ and a Hilbert space $\HH$, the weak time derivative $\partial_{t,\nu}$ on $\Lp[\nu]{2} (\R,\HH)$ is a densely defined, closed and normal operator. Moreover,
\begin{itemize}
\item $\partial_{t,\nu}=\LL^\ast_\nu (\iu\multm +\nu)\LL_\nu$ holds, where $\multm$ denotes the usual multiplication-by-argument operator on $\Lp{2} (\R,\HH)$ (i.e.\ $\multm f \coloneqq  [x\mapsto xf(x)]$),  and
\item for $\nu\neq 0$ the weak time derivative is one-to-one and onto with bounded inverse
\begin{equation*}
(\partial_{t,\nu})^{-1}g=\begin{cases}
t\mapsto\integral{-\infty}{t}{g(s)}{s}&\text{for }\nu>0,\\
t\mapsto-\integral{t}{\infty}{g(s)}{s}&\text{for }\nu<0,
\end{cases}
\end{equation*}
for $g\in \Lp[\nu]{2} (\R,\HH)$.
\end{itemize}
\end{lemma}

The coupling of our evolutionary system will be encoded via a functional calculus for $\partial_{t,\nu}$ in a \emph{material law}.
That is a holomorphic $M\colon U\to \Lb (\HH)$ where $\Lb(\HH)$ denotes the space of bounded linear operators on $\HH$, and $U\subseteq\C$ is open
and contains a half plane $\{z\in \C \mid \Real z\geq \nu\}\subseteq U$ for a $\nu\in\R$ on which $M$ is bounded, i.e.,
\begin{equation*}
\sup_{z\in \C, \Real z\geq \nu}\norm{M(z)}<\infty\text{.}
\end{equation*}
We write $\abscb (M)$ for the infimum over all such $\nu\in\R$. Corresponding to each material law, there exists the bounded \emph{material law operator}
$M(\partial_{t,\nu})\coloneqq \LL^\ast_\nu M(\iu\multm+\nu)\LL_\nu\in\Lb(\Lp[\nu]{2} (\R,\HH))$ where
\begin{equation*}
M(\iu\multm+\nu)\colon
\begin{cases}
\hfill \Lp{2} (\R,\HH) &\to\quad \Lp{2} (\R,\HH)\\
\hfill f&\mapsto\quad M(\iu\cdot+\nu)f(\cdot)
\end{cases}
\end{equation*}
 is a bounded operator.
 
 Finally, consider a spatial operator $A\colon\dom(A)\subseteq \HH\to \HH$ which is closed and densely defined. This operator, e.g., encodes the spatial
 differential operators of our system. Via
 \begin{equation}\label{eqDefinitionTensorTimeLiftingSpatialOperators}
 A_{\Lp[\nu]{2}}\colon
 \begin{cases}
\hfill \Lp[\nu]{2}(\R,\dom(A))\subseteq\Lp[\nu]{2} (\R,\HH) &\to\quad \Lp[\nu]{2} (\R,\HH)\\
\hfill f&\mapsto\quad [t\mapsto A(f(t))]
\end{cases}\text{,}
 \end{equation}
 $A$ can be lifted to a closed and densely defined operator on $\Lp[\nu]{2} (\R,\HH)$ with $(A_{\Lp[\nu]{2}})^{\ast}=(A^\ast)_{\Lp[\nu]{2}}$.
 We will abuse notation writing $A$ for $A_{\Lp[\nu]{2}}$ and $\dom (A)$ for $\Lp[\nu]{2}(\R,\dom(A))$.
 
 With all the necessary tools at hand, we are capable of defining an \emph{evolutionary equation}
 \begin{equation}\label{eqDefinEvolEq}
 (\partial_{t,\nu}M(\partial_{t,\nu})+A)u=f
 \end{equation}
 where $u,f\in \Lp[\nu]{2} (\R,\HH)$. Well-posedness would now mean that $\partial_{t,\nu}M(\partial_{t,\nu})+A$ is boundedly invertible. Unfortunately, the
 operator sum may not be closed. In fact, it already turns out not to be in very simple cases, e.g., the transport equation (cf.~\cite[Example~15.1.1]{SeTrWa22}). As a consequence, the following theorem (see, e.g., \cite[Theorem~6.2.1]{SeTrWa22})
 weakens~\labelcref{eqDefinEvolEq} by closing $\partial_{t,\nu}M(\partial_{t,\nu})+A$.
 \begin{thmPicard}\label{thmPicard}
 Consider a Hilbert space $\HH$, a material law $M\colon\dom(M)\subseteq \C\to \Lb (\HH)$ with
\begin{equation*}
\forall h\in \HH\forall z\in \C,\Real z\geq \nu_0:\Real\iprod[\HH]{h}{zM(z)h}\geq c\norm[\HH]{h}^2 
\end{equation*}
for a $c>0$ and a $\nu_0>\abscb (M)$, and a skew-selfadjoint $A\colon\dom(A)\subseteq \HH\to \HH$. Then,
\begin{itemize}
\item the closed subset $\overline{\partial_{t,\nu}M(\partial_{t,\nu})+A}$ of $\Lp[\nu]{2} (\R,\HH)\times \Lp[\nu]{2} (\R,\HH)$ is an (unbounded) operator for all $\nu\geq\nu_0$.
\item For $\nu\geq\nu_0$, $S_\nu\coloneqq (\overline{\partial_{t,\nu}M(\partial_{t,\nu})+A})^{-1}\in\Lb (\Lp[\nu]{2} (\R,\HH))$ with $\norm{S_\nu}\leq 1/c$.
\item $S_\nu$ is causal: $f\in \dom(S_\nu)=\Lp[\nu]{2} (\R,\HH)$ for $\nu\geq\nu_0$ with $\spt f\subseteq [a,\infty)$ for an $a\in\R$ implies $\spt S_\nu f\subseteq [a,\infty)$.
\item $S_\nu$ is eventually independent of $\nu$: $f\in \dom(S_{\nu_1})\cap\dom(S_{\nu_2})=\Lp[\nu_1]{2} (\R,\HH)\cap \Lp[\nu_2]{2} (\R,\HH)$ for $\nu_1,\nu_2\geq\nu_0$ implies $S_{\nu_1}f=S_{\nu_2}f$.
\item $f\in\dom(\partial_{t,\nu})$ for a $\nu\geq\nu_0$ implies $S_\nu f\in\dom(\partial_{t,\nu})\cap\dom(A)$, i.e., by $M(\partial_{t,\nu})\partial_{t,\nu}\subseteq \partial_{t,\nu}M(\partial_{t,\nu})$ we get a solution in the sense of~\labelcref{eqDefinEvolEq}.
\end{itemize}
 \end{thmPicard}
 
 In order to apply \Cref{thmPicard} to explicit examples, we need to specify $\HH$, $M$, and $A$. For our purposes, it will suffice to consider $A$ that arise from the
 following spatial differential operators.
 
 On an open subset $\Omega\subseteq\R^d$ for $d\in\N$, we have the usual spatial differential operators $\grad_{\Cc}\colon \Cc (\Omega)\subseteq \Lp{2}(\Omega)\to\Lp{2}(\Omega)^d$ and
 $\diver_{\Cc}\colon \Cc (\R)^d\subseteq \Lp{2}(\Omega)^d\to \Lp{2}(\Omega)$. In the case $d=3$, we can also consider $\curl_{\Cc}\colon \Cc (\Omega)^3\subseteq\Lp{2}(\Omega)^3\to\Lp{2}(\Omega)^3$. Via their $\Lp{2}$-adjoints, these can now be extended to their maximal $\Lp{2}$-domains in the usual weak/distributional sense: $\grad \coloneqq -(\diver_{\Cc})^\ast$, $\diver\coloneqq -(\grad_{\Cc})^\ast$ and
 $\curl\coloneqq (\curl_{\Cc})^\ast$. Adjoining again, we obtain the corresponding weak spatial differential operators with homogeneous boundary conditions:
 $\gradz \coloneqq -\diver^\ast=\overline{\grad_{\Cc}}$, $\diverz\coloneqq -\grad^\ast=\overline{\diver_{\Cc}}$ and $\curlz\coloneqq \curl^\ast=\overline{\curl_{\Cc}}$.
 \begin{example}\label{exEvoEq}
 Employing these spatial differential operators, we can provide some classical systems in the framework of evolutionary equations. In each case, we can immediately verify
 the conditions of \Cref{thmPicard}.
 \begin{itemize}
\item \emph{Heat Equation}\newline
For $\nu>0$, an open subset $\Omega\subseteq\R^d$ and $\HH\coloneqq
\Lp{2}(\Omega)\times \Lp{2}(\Omega)^d$,
the Heat Equation with heat source $Q\in\dom (\partial_{t,\nu})$, measurable and bounded
heat conductivity $a\colon\Omega\to \R^{d\times d}$ with
\begin{equation*}
\forall x\in\Omega:\Real a(x)\geq c
\end{equation*}
for a $c>0$ and thermally insulated boundary conditions reads
\begin{equation*}
\bigg(\partial_{t,\nu}\begin{pmatrix} 1&0\\ 0&0\end{pmatrix}+\begin{pmatrix} 0&0\\ 0&a^{-1}\end{pmatrix}
+\begin{pmatrix} 0&\diverz \\ \grad &0\end{pmatrix}\bigg)\begin{pmatrix} \vartheta\\ q\end{pmatrix}
=\begin{pmatrix} Q\\ 0\end{pmatrix}\text{,}
\end{equation*}
where $\vartheta$ is the heat distribution and $q$ the heat flux.

Formally, after elimination of $q$, this corresponds to the well-known standard form
\begin{equation*}
\dot{\vartheta}-\diverz a\grad \vartheta =Q\text{.}
\end{equation*}
\item \emph{Wave Equation}\newline
For $\nu>0$, an open subset $\Omega\subseteq\R^d$ and $\HH\coloneqq
\Lp{2}(\Omega)\times \Lp{2}(\Omega)^d$, the Wave Equation with balancing forces $f\in\dom(\partial_{t,\nu})$,
measurable and bounded elasticity tensor $T\colon\Omega\to \R^{d\times d}$ with
\begin{equation*}
\forall x\in\Omega: T(x)=T(x)^\ast\geq c
\end{equation*}
for a $c>0$ and clamped boundary conditions reads
\begin{equation*}
\bigg(\partial_{t,\nu}\begin{pmatrix} 1&0\\ 0&T^{-1}\end{pmatrix}
-\begin{pmatrix} 0&\diver \\ \gradz &0\end{pmatrix}\bigg)\begin{pmatrix} v\\ \sigma\end{pmatrix}
=\begin{pmatrix} f\\ 0\end{pmatrix}\text{,}
\end{equation*}
where $v\coloneqq \partial_{t,\nu} u$ for the displacement $u$ and $\sigma$ is the stress.

Formally, after elimination of $\sigma$, this corresponds to the well-known standard form
\begin{equation*}
\ddot{u}-\diver T\gradz u =f\text{.}
\end{equation*}
\item \emph{Maxwell's Equations}\newline
For an open subset $\Omega\subseteq\R^3$ and $\HH\coloneqq \Lp{2}(\Omega)^3\times \Lp{2}(\Omega)^3$,
assume that $c>0$ and $\nu_0>0$ exist such that the bounded and measurable dielectric permittivity,
magnetic permeability and electric conductivity $\varepsilon,\mu,\sigma\colon\Omega \to \R^{3\times 3}$
satisfy
\begin{equation*}
\forall x\in\Omega: \mu(x)=\mu(x)^\ast\geq c\text{ and } \nu\varepsilon(x)+\Real\sigma(x)=
\nu\varepsilon(x)^\ast+\Real\sigma(x)\geq c
\end{equation*}
for $\nu\geq \nu_0$. Then, Maxwell's equations with current $j_0\in\dom(\partial_{t,\nu})$ and perfect conductor
boundary conditions read
\begin{equation*}
\bigg(\partial_{t,\nu}\begin{pmatrix} \varepsilon&0\\ 0&\mu\end{pmatrix}+\begin{pmatrix} \sigma&0\\ 0&0\end{pmatrix}
+\begin{pmatrix} 0&-\curl \\ \curlz &0\end{pmatrix}\bigg)\begin{pmatrix} E\\ H\end{pmatrix}
=\begin{pmatrix} j_0\\ 0\end{pmatrix}\text{,}
\end{equation*}
where $E$ is the electric and $H$ the magnetic field. \qedhere
\end{itemize}
 \end{example}
 We see that, in a certain way, the skew-selfadjointness of $A$ asks for fulfilling boundary conditions. On the other hand, initial values are replaced with the weight $\nu$. That means they only appear implicitly in the general case. For specific classes of material laws, explicit initial conditions can be implemented in a distributional sense (see, e.g.,~\cite[Chapter~9]{SeTrWa22} and \Cref{thm:Solution_InitialValues}).
\subsection{Control Theory}

In this subsection, we shortly revisit some aspects of control theory for initial value problems. We restrict to the Hilbert space case; however, the theory is developed also in Banach spaces; cf.\ \cite{Douglas-66,DoleckiR-77,Lions1988,Carja1988}.

Let $X,U$ be Hilbert spaces, $A$ the generator of a strongly continuous semigroup $(S_t)_{t\geq 0}$ on $X$, $B\in \Lb(U,X)$ and $T>0$.
We consider the initial value problem
\begin{equation}\label{eq:Con}
  \begin{aligned}
    \dot{x}(t) & = Ax(t) + Bu(t) \quad(t\in (0,T]),\\
    x(0) & = x_0
\end{aligned}
\end{equation}
for some initial value $x_0\in X$.

We say that the system~\labelcref{eq:Con} is \emph{null-controllable} in time $T$ if for all $x_0\in X$ there exists $u\in \Lp{2}((0,T),U)$ such that for the solution $x$ we have $x(T) = 0$.

As it turns out (see \cite{Douglas-66,DoleckiR-77,Carja1988}) null-controllability in time $T$ is equivalent to an observability inequality for the so-called \emph{adjoint system}, which is given by
\begin{equation}
  \begin{aligned}\label{eq:Obs}
    \dot{\varphi}(t) & = -A^*\varphi(t)\quad(t\in [0,T)),\\
    \varphi(T) & = \varphi_T, \\
    \psi(t) & = B^*\varphi(t) \quad(t\in [0,T]),
\end{aligned}
\end{equation}
for $\varphi_T\in X$.
More precisely, \eqref{eq:Con} is null-controllable in time $T$ if and only if there exists $K\geq 0$ such that for all $\varphi_T\in X$ the solution $\varphi$ of \eqref{eq:Obs} satisfies
\begin{equation}\label{eq:ObsIneqAbsSett}
\norm[X]{\varphi(0)} \leq K \norm{\psi}_{\Lp{2}}.
\end{equation}
Note that \eqref{eq:Obs} is an equation backwards in time.
By means of the substitutions $\widetilde{\varphi}(t)\coloneqq \varphi(T-t)$ and $\widetilde{\psi}(t)\coloneqq  \psi(T-t)$ for $t\in [0,T]$ we can reformulate the adjoint system~\labelcref{eq:Obs} to an equation forward in time of the form
\begin{equation}
\begin{aligned}\label{eq:Obs2}
    \dot{\widetilde{\varphi}}(t) & = A^*\widetilde{\varphi}(t)\quad(t\in [0,T)),\\
    \widetilde{\varphi}(0) & = \widetilde{\varphi}_0, \\
    \widetilde{\psi}(t) & = B^*\widetilde{\varphi}(t) \quad(t\in [0,T]),
\end{aligned}
\end{equation}
for $\widetilde{\varphi}_0\in X$. The observability inequality~\labelcref{eq:ObsIneqAbsSett} then turns into a so-called \emph{final-state} observability inequality
\[\norm[X]{\widetilde{\varphi}(T)} \leq K \norm{\widetilde{\psi}}_{\Lp{2}}.\]

\begin{example}\label{exHGCont}
 We review the duality for the systems (without external sources) in Example \ref{exEvoEq}, formulated by means of semigroup theory.
 \begin{itemize}
\item The controlled {Heat Equation} can be written as
    \begin{align*}
        \dot{\vartheta}(t) & = \diver_0 a \grad \vartheta(t) + Bu(t) \quad(t\in (0,T]),\\
        \vartheta(0) & = \vartheta_0
    \end{align*}
    in $\Lp{2}(\Omega)$, where $\Omega\subseteq\R^d$ is open.
    Here, $A\coloneqq \diver_0 a \grad$ is the elliptic operator in $\Lp{2}(\Omega)$ with Neumann boundary conditions and $B\in \Lb(U,\Lp{2}(\Omega))$ and $U$ is some Hilbert space.
    Then, the adjoint system is given by
    \begin{align*}
        \dot{\varphi}(t) & = -\diver_0 a^* \grad \varphi(t) \quad(t\in [0,T))\\
        \varphi(T) & = \varphi_T,\\
        \psi(t) & = B^*\varphi(t) \quad (t\in [0,T]).
    \end{align*}
    Thus, the adjoint system is governed by a backward heat equation.
\item The controlled {Wave Equation} can be written as\footnote{We are using $[0,S]$ for the time interval in this example since $T$ already stands  for the elasticity.}
\begin{align*}
        \ddot{x}(t) & = \diver T \grad_0 x(t) + Bu(t) \quad(t\in (0,S]),\\
        x(0) & = x_0,\\
        \dot{x}(0) & = x_1,
    \end{align*}
    in $\Lp{2}(\Omega)$, where $\Omega\subseteq\R^d$ is open.
    Here, $A\coloneqq \diver T \grad_0$ is the elliptic operator in $\Lp{2}(\Omega)$ with Dirichlet boundary conditions and $B\in \Lb(U,\Lp{2}(\Omega))$ and $U$ is some Hilbert space.
    Introducing $v\coloneqq  \dot{x}$, we can reformulate this as a first-order system
    \begin{align*}
        \dot{\begin{pmatrix}x\\v\end{pmatrix}}(t) & = \begin{pmatrix} 0 & 1 \\ \diver T \grad_0  & 0 \end{pmatrix}\begin{pmatrix}x\\v\end{pmatrix}(t)  + \begin{pmatrix}0\\ Bu(t)\end{pmatrix} \quad(t\in (0,S]),\\
        \begin{pmatrix} x\\v\end{pmatrix}(0) & = \begin{pmatrix} x_0 \\x_1\end{pmatrix}.
    \end{align*}
    Then, the adjoint system is given by
    \begin{align*}
        \dot{\begin{pmatrix} \varphi\\\psi\end{pmatrix}}(t) & = -\begin{pmatrix} 0 & \diver T \grad_0 \\ 1 & 0 \end{pmatrix} \begin{pmatrix} \varphi\\\psi\end{pmatrix}(t) \quad(t\in [0,S))\\
        \begin{pmatrix} \varphi\\\psi\end{pmatrix}(S) & = \begin{pmatrix} \varphi_S\\\psi_S\end{pmatrix} \\
        \eta(t) & = B^*\psi(t) \quad (t\in [0,S]).
    \end{align*}
    Put differently, for $\psi$ we observe
    \[\ddot{\psi}(t) = \diver T \grad_0 \psi(t) \quad(t\in 
    [0,S)),\]
    so the adjoint system is governed by a wave equation.
    
\item The controlled {Maxwell's Equations} in the simplified situation with $\varepsilon=\mu=1$ can be written as
\begin{align*}
    \dot{\begin{pmatrix} E\\ H\end{pmatrix}}(t) & = \begin{pmatrix} -\sigma & \curl \\ -\curl_0  & 0 \end{pmatrix}\begin{pmatrix} E\\ H\end{pmatrix}(t)  + Bu(t) \quad(t\in (0,T]),\\
        \begin{pmatrix}  E\\ H\end{pmatrix}(0) & = \begin{pmatrix}  E_0 \\ H_0\end{pmatrix},
\end{align*}
in $\Lp{2}(\Omega)^3\times \Lp{2}(\Omega)^3$, where $\Omega\subseteq\R^3$ is open, $B \in \Lb(U,\Lp{2}(\Omega)^3\times \Lp{2}(\Omega)^3)$ for some Hilbert space $U$.
Then, the adjoint system is given by
    \begin{align*}
        \dot{\begin{pmatrix} \varphi\\\psi\end{pmatrix}}(t) & = -\begin{pmatrix} -\sigma^* & -\curl \\ \curl_0 & 0 \end{pmatrix} \begin{pmatrix} \varphi\\\psi\end{pmatrix}(t) \quad(t\in [0,T))\\
        \begin{pmatrix} \varphi\\\psi\end{pmatrix}(T) & = \begin{pmatrix} \varphi_T\\\psi_T\end{pmatrix} \\
        \eta(t) & = B^*\begin{pmatrix}\varphi\\\psi\end{pmatrix}(t) \quad (t\in [0,T]).
    \end{align*}
    This is again a Maxwell's equation with $\sigma$ replaced by $-\sigma^\ast$. \qedhere
\end{itemize}
 \end{example}

\section{Evolutionary Equations Backwards in Time}\label{secDefPropNuProd}
For a Hilbert space $\HH$, a material law $M\colon\dom(M)\subseteq \C\to \Lb (\HH)$ with
\begin{equation*}
\forall h\in \HH\forall z\in \C,\Real z\geq \nu_0:\Real\iprod[\HH]{h}{zM(z)h}\geq c\norm[\HH]{z}^2 
\end{equation*}
for a $c>0$ and a skew-selfadjoint $A\colon\dom(A)\subseteq \HH\to \HH$, we consider the evolutionary equation arising from the densely defined and closable operator (cf.~\labelcref{eqDefinEvolEq}
and~\Cref{thmPicard})
\begin{equation}\label{eqSecDefPropNuProd1}
\partial_{t,\nu}M(\partial_{t,\nu})+A=\LL^\ast_\nu (\iu\multm +\nu)M(\iu\multm+\nu)\LL_\nu+A
\end{equation}
where $\nu \geq \nu_0 > \max(\abscb (M),0)$ and $A$ is considered as its extension $A_{\Lp[\nu]{2}}$ to $\Lp[\nu]{2} (\R,\HH)$ (cf.~\labelcref{eqDefinitionTensorTimeLiftingSpatialOperators}).
Via the space $\Lp[\nu]{2}$, this operator itself implicitly only allows for solutions forward in time. Therefore, interpreting the adjoint system~\labelcref{eq:Obs}
as another evolutionary equation running backwards in time requires us to somehow get from $\Lp[\nu]{2}$ to $\Lp[-\nu]{2}$ while adjoining~\labelcref{eqSecDefPropNuProd1}.
\subsection{The \texorpdfstring{$\nu$}{nu}-Product}
Since $\exp(2\nu\multm)\colon \Lp[-\nu]{2} (\R,\HH)\to \Lp[\nu]{2} (\R,\HH)$
is a unitary map for a Hilbert space $\HH$ and $\nu\in\R$, we have
\begin{equation}\label{eq:ConnWeiInnProdNuProd}
\integral{\R}{}{\iprod[\HH]{f(t)}{g(t)}}{t}=\integral{\R}{}{\iprod[\HH]{f(t)}{\exp(2\nu t)g(t)}\euler^{-2\nu t}}{t}
\in\C
\end{equation}
for $f\in \Lp[\nu]{2} (\R,\HH)$ and $g\in \Lp[-\nu]{2} (\R,\HH)$.
\begin{definition}
We define the \emph{$\nu$-product}
\begin{equation*}
\dprod[\HH,\nu]{\cdot}{\cdot}\colon
\begin{cases}
\hfill \Lp[\nu]{2} (\R,\HH)\times \Lp[-\nu]{2} (\R,\HH)&\to\quad \C\\
\hfill (f,g)&\mapsto\quad \integral{\R}{}{\iprod[\HH]{f(t)}{g(t)}}{t}.
\end{cases} \qedhere
\end{equation*}
\end{definition}
If clear from the context, we will drop the index $\HH$. The $\nu$-product inherits many properties from the weighted $\Lp{2}$-products.
Note that the $\nu$-product was already implicitly used in~\cite{KPSTW14}.
\begin{remark}\label{remFirstPropNuProd}
Obviously, the $\nu$-product ist antilinear in the first and linear in the second argument, conjugate
symmetric in the sense that
\begin{equation*}
\overline{\dprod[\nu]{f}{g}}=\dprod[-\nu]{g}{f}
\end{equation*}
holds and we have a Cauchy--Schwarz inequality in the sense that
\begin{equation*}
\abs{\dprod[\nu]{f}{g}}=\bigabs{\integral{\R}{}{\iprod[\HH]{f(t)}{\exp(2\nu t)g(t)}\euler^{-2\nu t}}{t}}\leq \norm{f}_{\Lp[\nu]{2}}\norm{\exp(2\nu\multm)g}_{\Lp[\nu]{2}}=\norm{f}_{\Lp[\nu]{2}}\norm{g}_{\Lp[-\nu]{2}}
\end{equation*}
holds true for $f\in \Lp[\nu]{2} (\R,\HH)$ and $g\in \Lp[-\nu]{2} (\R,\HH)$.
\end{remark}
With that we also get an adapted version of the Riesz representation theorem.
\begin{lemma}
    Let $\HH$ be a Hilbert space and $\nu\in\R$. Consider the mappings
\begin{equation*}
\dprod[\nu]{f}{\cdot}\colon
\Lp[-\nu]{2} (\R,\HH)\to \C
\end{equation*}
and
\begin{equation*}
\dprod[\nu]{\cdot}{g}\colon \Lp[\nu]{2} (\R,\HH)\to \C
\end{equation*}
for $f\in \Lp[\nu]{2} (\R,\HH)$ and $g\in \Lp[-\nu]{2} (\R,\HH)$. They define (anti-)linear bounded functionals attaining their respective operator
norms $\norm{f}_{\Lp[\nu]{2}}$ and $\norm{g}_{\Lp[-\nu]{2}}$, i.e.,
    \begin{equation*}
    \norm{f}_{\Lp[\nu]{2}} = \max_{\substack{h\in \Lp[-\nu]{2} (\R,\HH)\\ \norm{h}_{\Lp[-\nu]{2}}\leq 1}} \abs{\dprod[\nu]{f}{h}}
    \qquad\text{and}\qquad
    \norm{g}_{\Lp[-\nu]{2}} = \max_{\substack{h\in \Lp[\nu]{2} (\R,\HH)\\ \norm{h}_{\Lp[\nu]{2}}\leq 1}} \abs{\dprod[\nu]{h}{g}}\text{.}
    \end{equation*}
Moreover, the mapping
\begin{equation*}
\Phi\colon\begin{cases}
\hfill \Lp[\nu]{2} (\R,\HH)&\to\quad \Lp[-\nu]{2} (\R,\HH)^\prime\\
\hfill f&\mapsto\quad \dprod[\nu]{f}{\cdot}
\end{cases}
\end{equation*}
is an antilinear bijective isometry.
\end{lemma}
\begin{proof}
The linearity and boundedness of $\dprod[\nu]{f}{.}$ with upper bound $\norm{f}_{\Lp[\nu]{2}}$ for the norm immediately follow from \Cref{remFirstPropNuProd}.
Testing with $\exp(-2\nu\multm)f$ and using that $\exp(-2\nu\multm)\colon \Lp[\nu]{2} (\R,\HH)\to \Lp[-\nu]{2} (\R,\HH)$ is unitary, we see 
\begin{equation*}
\dprod[\nu]{f}{\exp(-2\nu\multm)f}=\integral{\R}{}{\iprod[\HH]{f(t)}{f(t)}\euler^{-2\nu t}}{t}=\norm{f}_{\Lp[\nu]{2}}^2=\norm{f}_{\Lp[\nu]{2}}\norm{\exp(-2\nu\multm)f}_{\Lp[-\nu]{2}}\text{,}
\end{equation*}
i.e., the operator norm $\norm{f}_{\Lp[\nu]{2}}$ is attained. Analogously, the statements regarding $g$ are obtained. It remains to show that $\Phi$ is onto. This immediately follows from the Riesz representation theorem on $\Lp[-\nu]{2} (\R,\HH)$ since any element of $\Lp[-\nu]{2} (\R,\HH)^\prime$ is of the form
\begin{equation*}
\integral{\R}{}{\iprod[\HH]{g(t)}{.}\euler^{2\nu t}}{t}=\dprod[\nu]{\exp(2\nu\multm)g}{\cdot}=\Phi(\exp(2\nu\multm)g)\text{,}
\end{equation*}
where $g\in\Lp[-\nu]{2} (\R,\HH)$ and thus $\exp(2\nu\multm)g\in\Lp[\nu]{2} (\R,\HH)$.
\end{proof}
\begin{remark}
For $F\subseteq \Lp[\nu]{2} (\R,\HH)$ and $G\subseteq \Lp[-\nu]{2} (\R,\HH)$,
the continuity and the conjugate symmetry of the $\nu$-product imply that the annihilators (we will once again drop the index $\HH$ if clear from context)
\begin{gather*}
F^{\perp_{\HH,\nu}}\coloneqq \{g\in \Lp[-\nu]{2} (\R,\HH)\mid\forall f\in F:\dprod[\nu]{f}{g}=0\}=\bigcap_{f\in F}\ker\dprod[\nu]{f}{\cdot}
\subseteq \Lp[-\nu]{2} (\R,\HH)\\
\prescript{\perp_{\HH,\nu}}{}{G}\coloneqq\{f\in \Lp[\nu]{2} (\R,\HH)\mid\forall g\in G:\dprod[\nu]{f}{g}=0\}=\bigcap_{g\in G}
\ker\dprod[\nu]{\cdot}{g}\subseteq \Lp[\nu]{2} (\R,\HH)
\end{gather*}
are closed linear subspaces with $\overline{F}^{\perp_\nu}=F^{\perp_\nu}=\prescript{\perp_{-\nu}}{}{F}$ and $\prescript{\perp_\nu}{}{\overline{G}}=\prescript{\perp_\nu}{}{G}=
G^{\perp_{-\nu}}$. Moreover, we can easily verify
\begin{equation*}
F^{\perp_\nu}=\exp(-2\nu\multm)(F^{\perp})=(\exp(-2\nu\multm)F)^{\perp}
\end{equation*}
and thus get
\begin{equation*}
\big(\Lp[\nu]{2} (\R,\HH)\big)^{\perp_\nu}=\{0\}
\end{equation*}
as well as
\begin{equation*}\prescript{\perp_\nu}{}{(F^{\perp_\nu})}=(\exp(-2\nu\multm)(F^{\perp}))^{\perp_{-\nu}}=(F^\perp)^\perp=\cls (F)\text{.} \qedhere
\end{equation*}
\end{remark}

\subsection{The \texorpdfstring{$\nu$}{nu}-Adjoint}
The reason we introduced the $\nu$-product was to get from $\Lp[\nu]{2}$ to $\Lp[-\nu]{2}$ by adjoining. 
\begin{definition}
For $\nu\in\R$, Hilbert spaces $\HH_0,\HH_1$ and a relation $A\subseteq \Lp[\nu]{2} (\R,\HH_0)\times \Lp[\nu]{2} (\R,\HH_1)$, we define the \emph{$\nu$-adjoint}
\begin{equation}
A^{\ast_\nu}\coloneqq -\Big((A^{-1})^{\perp_{\HH_1\times \HH_0,\nu\times\nu}}\Big)\subseteq \Lp[-\nu]{2} (\R,\HH_1)\times \Lp[-\nu]{2} (\R,\HH_0)\text{,}\label{eqDefBIT}
\end{equation}
where we endowed $\Lp[\nu]{2} (\R,\HH_1)\times \Lp[\nu]{2} (\R,\HH_0)$ with $\dprod[\HH_1,\nu]{\cdot}{\cdot}+\dprod[\HH_0,\nu]{\cdot}{\cdot}$ and
$\perp_{\HH_1\times \HH_0,\nu\times\nu}$ stands for the respective annihilators.
\end{definition}
\begin{remark}\label{remAdjointBIT}
Reformulating~\labelcref{eqDefBIT}, we have
\begin{equation}\label{eqRemAdjointBIT1}
\begin{split}
A^{\ast_\nu}&=\{(g_1,g_2)\in \Lp[-\nu]{2} (\R,\HH_1)\times \Lp[-\nu]{2} (\R,\HH_0)\mid \forall (f_1,f_2)\in A:
\dprod[\HH_0,\nu]{f_1}{g_2}=\dprod[\HH_1,\nu]{f_2}{g_1}\}\\
&= (\exp(-2\nu\multm)\times \exp(-2\nu\multm))(A^{\ast})\\
&=((\exp(-2\nu\multm)\times \exp(-2\nu\multm))(A))^{\ast}\text{.}
\end{split}
\end{equation}
For a linear operator $A\subseteq \Lp[\nu]{2} (\R,\HH_0)\times \Lp[\nu]{2} (\R,\HH_1)$ and $g\in \Lp[-\nu]{2} (\R,\HH_1),h\in \Lp[-\nu]{2} (\R,\HH_0)$, we obtain
\begin{equation}\label{eqRemAdjointBIT2}
(g,h)\in A^{\ast_\nu}\iff \forall f\in\dom(A): \dprod[\HH_0,\nu]{f}{h}=\dprod[\HH_1,\nu]{A f}{g}\text{.} \qedhere
\end{equation}
\end{remark}
\begin{remark}\label{remSecondAdjointBIT}
Analogously to the classical Hilbert space adjoint, the $\nu$-adjoint has the following properties:
\begin{itemize}
\item In~\labelcref{eqDefBIT}, any order of application of $-$, $(\cdot)^{-1}$ and
$\perp_{\HH_1\times \HH_0,\nu\times\nu}$ or $\perp_{\HH_0\times \HH_1,\nu\times\nu}$ yields the same relation.
\item $A^{\ast_\nu}$ is a closed linear subspace and $A^{\ast_\nu}=(\overline{A})^{\ast_\nu}$ holds.
\item Identity~\labelcref{eqRemAdjointBIT1} implies
\begin{equation*}
\begin{split}
(A^{\ast_\nu})^{\ast_{-\nu}}&=\big((\exp(2\nu\multm)\times \exp(2\nu\multm))\big((\exp(-2\nu\multm)\times \exp(-2\nu\multm))(A^{\ast})\big)\big)^\ast\\
&=(A^\ast)^\ast=\cls(A)\text{.}
\end{split}
\end{equation*}
\end{itemize}
If $A$ even is a linear subspace, we can additionally infer from~\labelcref{eqRemAdjointBIT1}:
\begin{itemize}
\item $\ran (A)^{\perp_\nu}=\exp(-2\nu\multm)(\ran (A)^{\perp})=\exp(-2\nu\multm)(\ker(A^\ast))=\ker (A^{\ast_\nu})$
\item $\ker(\overline{A})^{\perp_\nu}=\exp(-2\nu\multm)(\ker(\overline{A})^{\perp})=\exp(-2\nu\multm)(\cran(A^\ast))=\cran(A^{\ast_\nu})$
\item $A^{\ast_\nu}$ is a linear operator if and only if $A$ is densely defined. If $A$ is a densely defined operator, then~\labelcref{eqRemAdjointBIT2} reads
\begin{equation}\label{eqRemSecondAdjointBIT1}
\dprod[\HH_0,\nu]{f}{A^{\ast_\nu}g}=\dprod[\HH_1,\nu]{A f}{g}
\end{equation}
for all $f\in\dom(A)$ and $g\in\dom(A^{\ast_\nu})$.
\end{itemize}
Once again, \labelcref{eqRemAdjointBIT1}  shows that a linear operator $A$ is closable if and only if $A^{\ast_\nu}$ is densely defined,
and for a linear $A$ that the closure $\overline{A}$ is a bounded operator (including $\dom(\overline{A})=\Lp[\nu]{2} (\R,\HH)$) if and only if $A^{\ast_\nu}$ is.
In the latter case, $\norm{\overline{A}}=\norm{A^{\ast_\nu}}$ holds.
\end{remark}
We close with a few more easy technical lemmas that we will use in order to compute the $\nu$-adjoints of concrete operators.
\begin{lemma}
    Consider $\nu\in\R$ and Hilbert spaces $\HH_0,\HH_1,\HH_2$. For $A,C\in\Lb\bigl(\Lp[\nu]{2}(\R,\HH_0), \Lp[\nu]{2}(\R,\HH_1)\bigr)$ and $B\in\Lb\bigl(\Lp[\nu]{2}(\R,\HH_1), \Lp[\nu]{2}(\R,\HH_2)\bigr)$, we have $(BA)^{\ast_\nu} = A^{\ast_\nu}B^{\ast_\nu}$ and $(A+C)^{\ast_\nu}=A^{\ast_\nu}+C^{\ast_\nu}$.
\end{lemma}

\begin{proof}
This immediately follows from~\labelcref{eqRemSecondAdjointBIT1}, since $A, B, C, BA, A + C$ and their respective $\nu$-adjoints are bounded by \Cref{remSecondAdjointBIT}.
\end{proof}

\begin{lemma}\label{lemmaSumProdOfNuAdjRel}
    Consider $\nu\in\R$ and Hilbert spaces $\HH_0,\HH_1,\HH_2$. For $A,C\subseteq \Lp[\nu]{2}(\R,\HH_0)\times \Lp[\nu]{2}(\R,\HH_1)$ and $B\subseteq\Lp[\nu]{2}(\R,\HH_1)\times \Lp[\nu]{2}(\R,\HH_2)$, we have $(BA)^{\ast_\nu} \supseteq
     \overline{A^{\ast_\nu}B^{\ast_\nu}}$ and $(A+C)^{\ast_\nu}\supseteq \overline{A^{\ast_\nu}+C^{\ast_\nu}}$. If $B$ is a closed linear subspace and $A\in \Lb\bigl(\Lp[\nu]{2}(\R,\HH_0), \Lp[\nu]{2}(\R,\HH_1)\bigr)$,
     then we have $(BA)^{\ast_\nu}=\overline{A^{\ast_\nu}B^{\ast_\nu}}$. Alternatively, if $B\in \Lb\bigl(\Lp[\nu]{2}(\R,\HH_1), \Lp[\nu]{2}(\R,\HH_2)\bigr)$ and $A$ either is a closed linear subspace or
      a densely defined operator,
     then we have $(BA)^{\ast_\nu}=A^{\ast_\nu}B^{\ast_\nu}$.
\end{lemma}

\begin{proof}
The first two statements immediately follow from~\labelcref{eqRemAdjointBIT1} and the $\nu$-adjoint being closed. For the third statement, one can readily show $BA=(A^{\ast_\nu}B^{\ast_\nu})^{\ast_{-\nu}}$
using the boundedness of $A$ and $A^{\ast_\nu}$, $(B^{\ast_\nu})^{\ast_{-\nu}}=B$ and~\labelcref{eqRemAdjointBIT1}. For a bounded operator $B$ and
a closed linear subspace $A$, we can compute $(A^{\ast_\nu}B^{\ast_\nu})^{\ast_{-\nu}}=\overline{BA}$ by virtue of the third statement. For a bounded operator $B$ and
a densely defined operator $A$, one can use the boundedness of $B$ and $B^{\ast_\nu}$, $\dom(BA)=\dom(A)$ and~\labelcref{eqRemSecondAdjointBIT1} to easily show that
$f\in \dom((BA)^{\ast_\nu})$ implies $B^{\ast_\nu} f\in\dom(A^{\ast_\nu}) \iff f\in\dom(A^{\ast_\nu}B^{\ast_\nu})$.
\end{proof}

%
\begin{lemma}\label{lemmaNuAdjOfExtendedSpatialOp}
Consider Hilbert spaces $\HH_0,\HH_1$. Let $A\colon\dom(A)\subseteq \HH_0\to \HH_1$ be densely defined and closed.
Then for $\nu\in\R$, its extension (cf.~\labelcref{eqDefinitionTensorTimeLiftingSpatialOperators}) to $\Lp[\nu]{2} (\R, \HH_{0})$ has the extension of its adjoint to $\Lp[-\nu]{2} (\R, \HH_1)$ as its $\nu$-adjoint, i.e., $(A_{\Lp[\nu]{2}})^{\ast_\nu}=(A^\ast)_{\Lp[-\nu]{2}}$.
\end{lemma}
\begin{proof}
From~\Cref{remAdjointBIT} and $(A_{\Lp[\nu]{2}})^{\ast}=(A^\ast)_{\Lp[\nu]{2}}$, we can infer
\begin{equation}
\begin{split}
(A_{\Lp[\nu]{2}})^{\ast_\nu}&=(\exp(-2\nu\multm)\times \exp(-2\nu\multm))((A_{\Lp[\nu]{2}})^\ast)\\
&=(\exp(-2\nu\multm)\times \exp(-2\nu\multm))((A^\ast)_{\Lp[\nu]{2}})\\
&=(A^\ast)_{\Lp[-\nu]{2}}\text{.}
\end{split}\label{eqBITAdjSAdj}
\end{equation}
The last equality is a consequence of $\exp(-2\nu\multm)(\Lp[\nu]{2}(\R,\dom(A^\ast)))=\Lp[-\nu]{2}(\R,\dom(A^\ast))$.
\end{proof}
\subsection{The \texorpdfstring{$\nu$}{nu}-Adjoint of Material Laws}
 We first want to compute the $\nu$-adjoint of $\partial_{t,\nu}M(\partial_{t,\nu})=\LL^\ast_\nu (\iu\multm +\nu)M(\iu\multm+\nu)\LL_\nu$ in~\labelcref{eqSecDefPropNuProd1}.
\begin{lemma}\label{lemmaBITAdjViaR}
  For $\nu\in\R$, a Hilbert space $\HH$ and a linear operator $R\subseteq \Lp[\nu]{2} (\R,\HH)\times \Lp[\nu]{2} (\R,\HH)$, we have
  \begin{equation*}
R^{\ast_\nu}=\LL^\ast_{-\nu}(\LL_\nu R\LL^\ast_\nu)^\ast\LL_{-\nu}\text{.}
\end{equation*}
\end{lemma}
\begin{proof}
As $\LL^\ast_\nu=\LL^{-1}_\nu$ holds, we have $(g_1,g_2)\in R^{\ast_\nu}$ if and only if
\begin{equation}
\dprod[\nu]{\LL^\ast_\nu\LL_\nu f}{g_2}=\dprod[\nu]{f}{g_2}=\dprod[\nu]{Rf}{g_1}
=\dprod[\nu]{\LL^\ast_\nu\LL_\nu R\LL^\ast_\nu\LL_\nu f}{g_1}
\label{eqAdjMatLaw1}
\end{equation}
for all $f\in\dom (R)$. For any $h_1\in \Lp{2} (\R,\HH)$ and $h_2\in \Lp[-\nu]{2} (\R,\HH)$, we have
\begin{equation*}
\begin{split}
\dprod[\nu]{\LL^\ast_\nu h_1}{h_2}&=\iprod[{\Lp[\nu]{2}}]{\LL^\ast_\nu h_1}{\exp(2\nu\multm)h_2}\\
&=\iprod[\Lp{2}]{h_1}{\LL_\nu\exp(2\nu\multm)h_2}\\
&=\iprod[\Lp{2}]{h_1}{\LL_{-\nu}h_2}\text{.}
\end{split}
\end{equation*}
The last identity $\LL_{-\nu}=\LL_\nu\exp(2\nu\multm)$ follows from
\begin{equation*}
\LL_{-\nu}=\FF\exp(\nu\multm)=\FF\exp(-\nu\multm)\exp(2\nu\multm)=\LL_\nu\exp(2\nu\multm)\text{.}
\end{equation*}
Hence,~\labelcref{eqAdjMatLaw1} turns into
\begin{equation}
\iprod[\Lp{2}]{h}{\LL_{-\nu}g_2}=\iprod[\Lp{2}]{\LL_\nu R\LL^\ast_\nu h}{\LL_{-\nu}g_1}\label{eqAdjMatLaw2}
\end{equation}
for all $h\in \LL_\nu\dom(R)=\dom(\LL_\nu R\LL^\ast_\nu)$. By definition,~\labelcref{eqAdjMatLaw2}
is true if and only if $(\LL_{-\nu}g_1,\LL_{-\nu}g_2)\in (\LL_\nu R\LL^\ast_\nu)^\ast$ holds.
\end{proof}
\begin{definition}
  For a Hilbert space $\HH$ and a material law $M\colon\dom(M)\subseteq \C\to \Lb (\HH)$, we define
\begin{equation*}
M^\ast\colon
\begin{cases}
\hfill \dom(M)\subseteq\C&\to\quad \Lb(\HH)\\
\hfill z&\mapsto\quad M(z)^\ast. 
\end{cases}\qedhere
\end{equation*}
\end{definition}
\begin{remark}\label{remDualMatLawProperties}
  In a certain way, $M^{\ast}$ defines a \enquote{dual} material law with $\abscb(M^{\ast})=\abscb(M)$
as well as $\norm[\infty,\C_{\Real>\nu}]{M}=\norm[\infty,\C_{\Real>\nu}]{M^{\ast}}$ for all $\nu>\abscb(M)$. Note that in general $M^\ast$ will not be holomorphic.
Moreover, we immediately get $(M(\iu\multm+\nu))^{\ast}=M^\ast (\iu\multm+\nu)$ for such $\nu$. Thus, \Cref{lemmaBITAdjViaR} yields
\begin{equation*}
M(\partial_{t,\nu})^{\ast_\nu}=\LL^\ast_{-\nu}(M(\iu\multm+\nu))^\ast\LL_{-\nu}=\LL^\ast_{-\nu}M^\ast(\iu\multm+\nu)\LL_{-\nu}\text{.} \qedhere
\end{equation*}
\end{remark}
\begin{lemma}\label{lemmaAdjTimeDerivMatLaw}
  Consider a Hilbert space $\HH$, a material law $M\colon\dom(M)\subseteq \C\to \Lb (\HH)$ and $\nu>\max(\abscb (M),0)$.
  Then,
  \begin{equation*}
(\partial_{t,\nu}M(\partial_{t,\nu}))^{\ast_\nu}=-\partial_{t,-\nu}\LL^\ast_{-\nu}M^\ast(\iu\multm+\nu)\LL_{-\nu}
  \end{equation*}
  holds.
\end{lemma}
\begin{proof}
  Boundedness of $M^\ast (\iu\multm+\nu)$ shows $M^\ast (\iu\multm+\nu)(\iu\multm -\nu)\subseteq (\iu\multm -\nu)M^\ast (\iu\multm+\nu)$ with
$\dom (M^\ast (\iu\multm+\nu)(\iu\multm -\nu))=
\dom ((\iu\multm -\nu))$. Consider the closed operator
\begin{equation*}
\LL^\ast_{-\nu}(\iu\multm -\nu)M^\ast(\iu\multm+\nu)\LL_{-\nu}= \partial_{t,-\nu}\LL^\ast_{-\nu}M^\ast(\iu\multm+\nu)\LL_{-\nu}
\end{equation*}
and the operator
\begin{equation*}
\LL^\ast_{-\nu}M^\ast(\iu\multm+\nu)(\iu\multm -\nu)\LL_{-\nu}= \LL^\ast_{-\nu}M^\ast(\iu\multm+\nu)\LL_{-\nu}\partial_{t,-\nu}\text{.}
\end{equation*}
We then have
\begin{equation}\label{eqSubsetMatLawTD}
\LL^\ast_{-\nu}M^\ast(\iu\multm+\nu)\LL_{-\nu}\partial_{t,-\nu}\subseteq \partial_{t,-\nu}\LL^\ast_{-\nu}M^\ast(\iu\multm+\nu)\LL_{-\nu}
\end{equation}
(the domain of the left operator is $\dom (\partial_{t,-\nu})$). This implies
\begin{equation}\label{eqSubsetMatLawResTD}
\LL^\ast_{-\nu}M^\ast(\iu\multm+\nu)\LL_{-\nu}(1-\varepsilon\partial_{t,-\nu})^{-1}= (1-\varepsilon\partial_{t,-\nu})^{-1}\LL^\ast_{-\nu}M^\ast(\iu\multm+\nu)\LL_{-\nu}
\end{equation}
for $\varepsilon > 0$ since $\LL^\ast_{-\nu}M^\ast(\iu\multm+\nu)\LL_{-\nu}$ is bounded.
For an element $g\in\dom(\partial_{t,-\nu}\LL^\ast_{-\nu}M^\ast(\iu\multm+\nu)\LL_{-\nu})$ of the domain of the operator on the righthand side, we define $g_{\varepsilon}\coloneqq (1-\varepsilon\partial_{t,-\nu})^{-1}g$
for $\varepsilon > 0$ (cf.\ the Yosida-approximation approach in~\cite[proof of Lemma~16.3.3]{SeTrWa22}). Hence, we have $g_\varepsilon\xrightarrow{\varepsilon\to 0+} g$ in $\Lp[-\nu]{2} (\R,\HH)$ and
 $g_{\varepsilon}\in\dom (\partial_{t,-\nu})$ holds. Therefore,~\labelcref{eqSubsetMatLawTD} and~\labelcref{eqSubsetMatLawResTD}  yield
\begin{align*}
\LL^\ast_{-\nu}M^\ast(\iu\multm+\nu)\LL_{-\nu}\partial_{t,-\nu}g_\varepsilon&=\partial_{t,-\nu}\LL^\ast_{-\nu}M^\ast(\iu\multm+\nu)\LL_{-\nu}g_\varepsilon\\
&=\partial_{t,-\nu}\LL^\ast_{-\nu}M^\ast(\iu\multm+\nu)\LL_{-\nu} (1-\varepsilon\partial_{t,-\nu})^{-1}g\\
&=\partial_{t,-\nu}(1-\varepsilon\partial_{t,-\nu})^{-1}\LL^\ast_{-\nu}M^\ast(\iu\multm+\nu)\LL_{-\nu}g\text{.}
\intertext{As an (unbounded) operator and its resolvent commute on the domain of the operator and as
$g\in\dom(\partial_{t,-\nu}\LL^\ast_{-\nu}M^\ast(\iu\multm+\nu)\LL_{-\nu})$ implies $\LL^\ast_{-\nu}M^\ast(\iu\multm+\nu)\LL_{-\nu}g\in\dom (\partial_{t,-\nu})$, we conclude}
&=(1-\varepsilon\partial_{t,-\nu})^{-1}\partial_{t,-\nu}\LL^\ast_{-\nu}M^\ast(\iu\multm+\nu)\LL_{-\nu}g\\
&\xrightarrow{\varepsilon\to 0+}\partial_{t,-\nu}\LL^\ast_{-\nu}M^\ast(\iu\multm+\nu)\LL_{-\nu}g
\end{align*}
in $\Lp[-\nu]{2} (\R,\HH)$, i.e., we have shown
\begin{equation}\label{eqMatLawTDCommClos}
\overline{\LL^\ast_{-\nu}M^\ast(\iu\multm+\nu)\LL_{-\nu}\partial_{t,-\nu}}= \partial_{t,-\nu}\LL^\ast_{-\nu}M^\ast(\iu\multm+\nu)\LL_{-\nu}\text{.}
\end{equation}
Applying \Cref{lemmaBITAdjViaR} to $\partial_{t,\nu}M(\partial_{t,\nu})$ and using~\labelcref{eqMatLawTDCommClos}, we infer
\begin{equation*}
\begin{split}
(\partial_{t,\nu}M(\partial_{t,\nu}))^{\ast_\nu}&=\LL^\ast_{-\nu}((\iu\multm +\nu)M(\iu\multm+\nu))^\ast\LL_{-\nu}\\
&=\LL^\ast_{-\nu}\overline{M^\ast(\iu\multm+\nu)(-\iu\multm +\nu)}\LL_{-\nu}\\
&=\overline{-\LL^\ast_{-\nu}M^\ast(\iu\multm+\nu)(\iu\multm -\nu)\LL_{-\nu}}\\
&=-\partial_{t,-\nu}\LL^\ast_{-\nu}M^\ast(\iu\multm+\nu)\LL_{-\nu}\text{.} \qedhere
\end{split}
\end{equation*}
\end{proof}
\begin{remark}\label{remNuAdjOfTimeDeriv}
If we take the constant identity as the material law, \Cref{lemmaAdjTimeDerivMatLaw} reads $(\partial_{t,\nu})^{\ast_\nu}=-\partial_{t,-\nu}$ for any Hilbert space $\HH$ and $\nu>0$, see also~\cite[Remark~2.9]{KPSTW14}.
\end{remark}
\subsection{The \texorpdfstring{$\nu$}{nu}-Adjoint System}
We are now able to compute the $\nu$-adjoint of our operator~\labelcref{eqSecDefPropNuProd1}.
\begin{remark}
Applying \Cref{lemmaNuAdjOfExtendedSpatialOp} to the skew-selfadjoint $A$ in~\labelcref{eqSecDefPropNuProd1}, we have $(A_{\Lp[\nu]{2}})^{\ast_\nu}=-(A)_{\Lp[-\nu]{2}}$.
Abusing notation, we will write $A^{\ast_\nu}=-A$ and we will denominate the domain of $(A_{\Lp[\nu]{2}})^{\ast_\nu}$ as $\dom(A^{\ast_\nu})=\dom(A)$.
Combining this and \Cref{lemmaAdjTimeDerivMatLaw} via \Cref{lemmaSumProdOfNuAdjRel}, we conclude
\begin{equation}
-\partial_{t,-\nu}\LL^\ast_{-\nu}M^\ast(\iu\multm+\nu)\LL_{-\nu}-A\subseteq (\partial_{t,\nu}M(\partial_{t,\nu})+A)^{\ast_\nu}\text{,}
~\label{eqBITSystemLeq}
\end{equation}
and the domain of the operator on the left side is (cf.~\labelcref{eqSubsetMatLawTD})
\begin{equation}\label{eqSubsetDomainEvEqNuAdj}
\dom (\partial_{t,-\nu}\LL^\ast_{-\nu}M^\ast(\iu\multm+\nu)\LL_{-\nu})\cap \dom (A)\supseteq \dom(\partial_{t,-\nu})\cap \dom (A)\text{.} \qedhere
\end{equation}
\end{remark}
Unfortunately, the spatial operator $A$ usually does not render the operator $-\partial_{t,-\nu}\LL^\ast_{-\nu}M^\ast(\iu\multm+\nu)\LL_{-\nu}-A$ closed. This
is for the same reason that already appeared in \Cref{thmPicard}. In other words, the best improvement possible for~\labelcref{eqBITSystemLeq}
is the following result based on~\cite[proof of Lemma~16.3.4]{SeTrWa22}.
\begin{theorem}\label{thmBITNuAdjSys}
Consider a Hilbert space $\HH$, a material law $M\colon\dom(M)\subseteq \C\to \Lb (\HH)$ with $\nu>\max(\abscb (M),0)$, as well as a skew-selfadjoint $A\colon\dom(A)\subseteq \HH\to \HH$.
Then, we have
\begin{equation}\label{eqBITSystemEqual}
\overline{-\partial_{t,-\nu}\LL^\ast_{-\nu}M^\ast(\iu\multm+\nu)\LL_{-\nu}-A}= (\partial_{t,\nu}M(\partial_{t,\nu})+A)^{\ast_\nu}\text{,}
\end{equation}
and~\labelcref{eqBITSystemEqual} is a linear operator.
\end{theorem}
\begin{proof}
First, note that $(\partial_{t,\nu}M(\partial_{t,\nu})+A)^{\ast_\nu}$ is an operator since $\partial_{t,\nu}M(\partial_{t,\nu})+A$ is densely defined.

For $g \in \dom ((\partial_{t,\nu}M(\partial_{t,\nu})+A)^{\ast_\nu})$, define $g_{\varepsilon}\coloneqq (1-\varepsilon\partial_{t,-\nu})^{-1}g$
for $\varepsilon > 0$ as in \Cref{lemmaAdjTimeDerivMatLaw} again.
Applying \Cref{lemmaBITAdjViaR} to $1+\varepsilon\partial_{t,\nu}=1+\varepsilon\LL^\ast_\nu(\iu\multm +\nu)\LL_\nu$ yields\footnote{Mind that the boundedness of $1$ and
the closedness of $\partial_{t,\nu}$ allow for the sum and the adjoint to interchange.}
\begin{align*}
(1+\varepsilon\partial_{t,\nu})^{\ast_\nu} &= (1+\varepsilon\LL^\ast_\nu(\iu\multm +\nu)\LL_\nu)^{\ast_\nu}\\
&=1+\varepsilon\LL^\ast_{-\nu}(\iu\multm +\nu)^\ast\LL_{-\nu}\\
&=1-\varepsilon\LL^\ast_{-\nu}(\iu\multm -\nu)\LL_{-\nu}\\
&=1-\varepsilon\partial_{t,-\nu}\text{.}
\end{align*}
 For $f\in\dom(A)$,
the theorem of Hille for Bochner integrals and the properties of $(\partial_{t,\nu})^{-1}$ (cf.~\Cref{lemmaWeakTimeDerProperties})
imply 
\begin{equation}
(\partial_{t,\nu})^{-1}Af=\Bigl[t\mapsto\integral{-\infty}{t}{Af(s)}{s}\Bigr]=\Bigl[t\mapsto A\integral{-\infty}{t}{f(s)}{s}\Bigr]=A(\partial_{t,\nu})^{-1}f.\label{eq:NotSoRigAINt}
\end{equation}
This yields $(\partial_{t,\nu})^{-1}A\subseteq A(\partial_{t,\nu})^{-1}$ and consequently, for $\varepsilon>0$,
\begin{equation*}
(1+\varepsilon\partial_{t,\nu})^{-1}A\subseteq A(1+\varepsilon\partial_{t,\nu})^{-1}\text{.}
\end{equation*}
Arguing similarly to~\labelcref{eqSubsetMatLawTD} and~\labelcref{eqSubsetMatLawResTD},
we can show $M(\partial_{t,\nu})\partial_{t,\nu}\subseteq\partial_{t,\nu}M(\partial_{t,\nu})$, and thus we also have
$(1+\varepsilon\partial_{t,\nu})^{-1}M(\partial_{t,\nu})=M(\partial_{t,\nu})(1+\varepsilon\partial_{t,\nu})^{-1}$. Clearly, $(1-\varepsilon\partial_{t,-\nu})^{-1}\partial_{t,-\nu}\subseteq
\partial_{t,-\nu}(1-\varepsilon\partial_{t,-\nu})^{-1}$ holds. Moreover, $f\in \dom (\partial_{t,\nu}M(\partial_{t,\nu}))$
implies $M(\partial_{t,\nu})f\in \dom (\partial_{t,\nu})$. Hence, $(1+\varepsilon\partial_{t,\nu})^{-1}\partial_{t,\nu}M(\partial_{t,\nu})f=
\partial_{t,\nu}(1+\varepsilon\partial_{t,\nu})^{-1}M(\partial_{t,\nu})f$ holds for such $f$.
Therefore,
\begin{equation}
\begin{split}
\dprod[\nu]{(\partial_{t,\nu}M(\partial_{t,\nu})+A)f}{g_\varepsilon}&=\dprod[\nu]{(1+\varepsilon\partial_{t,\nu})^{-1}(\partial_{t,\nu}M(\partial_{t,\nu})+A)f}{g}\\
&=\dprod[\nu]{(\partial_{t,\nu}M(\partial_{t,\nu})+A)(1+\varepsilon\partial_{t,\nu})^{-1}f}{g}\\
&=\dprod[\nu]{f}{(1-\varepsilon\partial_{t,-\nu})^{-1}(\partial_{t,\nu}M(\partial_{t,\nu})+A)^{\ast_\nu}g}\\
\end{split}\label{eqResApproxInDom}
\end{equation}
holds true for $f\in \dom (\partial_{t,\nu}M(\partial_{t,\nu}))\cap \dom (A)$. In other words, we have shown that
$g_\varepsilon\in\dom ((\partial_{t,\nu}M(\partial_{t,\nu})+A)^{\ast_\nu})$ and
\begin{equation*}
(\partial_{t,\nu}M(\partial_{t,\nu})+A)^{\ast_\nu}g_\varepsilon=(1-\varepsilon\partial_{t,-\nu})^{-1}(\partial_{t,\nu}M(\partial_{t,\nu})+A)^{\ast_\nu}g
\xrightarrow{\varepsilon\to 0+} (\partial_{t,\nu}M(\partial_{t,\nu})+A)^{\ast_\nu}g\text{.}
\end{equation*}
Rewriting~\labelcref{eqResApproxInDom} yields
\begin{align*}
\dprod[\nu]{Af}{g_\varepsilon}&=\dprod[\nu]{f}{(1-\varepsilon\partial_{t,-\nu})^{-1}(\partial_{t,\nu}M(\partial_{t,\nu})+A)^{\ast_\nu}g}
-\dprod[\nu]{\partial_{t,\nu}M(\partial_{t,\nu})f}{g_\varepsilon}\\
&=\dprod[\nu]{f}{(1-\varepsilon\partial_{t,-\nu})^{-1}(\partial_{t,\nu}M(\partial_{t,\nu})+A)^{\ast_\nu}g}\\
&\quad +\dprod[\nu]{f}{\partial_{t,-\nu}\LL^\ast_{-\nu}M^\ast(\iu\multm+\nu)\LL_{-\nu}g_\varepsilon}
\end{align*}
for $f\in \dom (\partial_{t,\nu})\cap \dom (A)$. As the $\nu$-product is bounded and  $\dom (\partial_{t,\nu})\cap \dom (A)$ is a
dense subset of $\dom (A)$,\footnote{For $f\in \dom (A)$, we have $(1+\varepsilon\partial_{t,\nu})^{-1}f\xrightarrow{\varepsilon\to 0+} f$ and 
$A(1+\varepsilon\partial_{t,\nu})^{-1}f=(1+\varepsilon\partial_{t,\nu})^{-1}Af\xrightarrow{\varepsilon\to 0+} Af$.}
 we conclude $g_\varepsilon\in\dom (A^{\ast_\nu})=\dom (A)$. To sum up, we have proven that $g_\varepsilon$ converges
to $g$ in the domain of the $\nu$-adjoint, that the image of $g_\varepsilon$ converges to the image of $g$ and that $g_\varepsilon$ even
lies in the domain of the operator on the left side of~\labelcref{eqBITSystemLeq}.
\end{proof}
\begin{definition}
Taking \Cref{thmBITNuAdjSys} into consideration in the setting of~\labelcref{eqSecDefPropNuProd1}, we will call
\begin{equation}\label{definNuAdjSysclassEq}
(-\partial_{t,-\nu}\LL^\ast_{-\nu}M^\ast(\iu\multm+\nu)\LL_{-\nu}-A)u=f
\end{equation}
for suitable $u,f\in \Lp[-\nu]{2}(\R,\HH)$ (cf.\ \Cref{thmMainWellPosednessNuAdj}) the \emph{$\nu$-adjoint system} corresponding to the spatial operator $A$ and the material law $M$.
\end{definition}
\begin{remark}
In the special case
\begin{equation}
M(z)\coloneqq\sum_{k=0}^n z^{-k}M_k\label{eqMatLawFinSumOfBdOp}
\end{equation}
for $n\in\N\uplus\{0\}$, $z\in\C\setminus\{0\}$ and $M_0,M_1,\dots,M_n\in \Lb(\HH)$ we can compute
\begin{equation}
\begin{split}
\LL^\ast_{-\nu}M^\ast(\iu\multm+\nu)\LL_{-\nu}&=\LL^\ast_{-\nu}\Big(\sum_{k=0}^n (-\iu\multm+\nu)^{-k}M^\ast_k\Big)\LL_{-\nu}\\
&=\sum_{k=0}^n (-1)^k \LL^\ast_{-\nu}(\iu\multm-\nu)^{-k}\LL_{-\nu}M^\ast_k\\
&=\sum_{k=0}^n(-1)^k \partial_{t,-\nu}^{-k}M^\ast_k\text{.}
\end{split}\label{eqBITAdjOfFinSum}
\end{equation}
Here, the first equality follows from $(\iu\multm+\nu)^{-1}$ being bounded and the second one since the theorem of Hille (applied to bounded
operators) implies
\begin{equation*}
(M_k^\ast)_{\Lp{2}}\LL_{-\nu}\varphi=\LL_{-\nu}(M_k^\ast)_{{\Lp[-\nu]{2}}}\varphi
\end{equation*}
for all $\varphi$ from the dense subset $\Cc(\R,\HH)$ of $\Lp[-\nu]{2}(\R,\HH)$.
\end{remark}
\begin{example}\label{exNuAdjOfEvoEqEx}
Applying~\labelcref{eqBITAdjOfFinSum} to \Cref{exEvoEq}, we can provide a few examples:
\begin{itemize}
\item \emph{Heat Equation}\newline
The differential operator of the $\nu$-adjoint heat equation reads
\begin{equation*}
-\partial_{t,-\nu}\begin{pmatrix} 1&0\\ 0&0\end{pmatrix}+\begin{pmatrix} 0&0\\ 0&a^{-\ast}\end{pmatrix}
-\begin{pmatrix} 0&\diverz \\ \grad &0\end{pmatrix}\text{.}
\end{equation*}
\item \emph{Wave Equation}\newline
The differential operator of the $\nu$-adjoint wave equation reads
\begin{equation*}
-\partial_{t,-\nu}\begin{pmatrix} 1&0\\ 0&T^{-1}\end{pmatrix}
+\begin{pmatrix} 0&\diver \\ \gradz &0\end{pmatrix}\text{.}
\end{equation*}
\item \emph{Maxwell's Equations}\newline
The differential operator of the $\nu$-adjoint Maxwell's equations reads
\begin{equation*}
-\partial_{t,-\nu}\begin{pmatrix} \varepsilon&0\\ 0&\mu\end{pmatrix}+\begin{pmatrix} \sigma^\ast&0\\ 0&0\end{pmatrix}
-\begin{pmatrix} 0&-\curl \\ \curlz &0\end{pmatrix}\text{.}
\end{equation*}
\end{itemize}
Comparing with Example \ref{exHGCont} (for the cases $B=0$) we obtain that these $\nu$-adjoint systems are consistent with the adjoint systems in Example \ref{exHGCont}.
\end{example}
\subsection{Time-Reversal}\label{secTimeReversalOfNuAdjoint}
Instead of treating the plain adjoint system~\labelcref{eq:Obs}, it is common to apply a time shift, i.e., to let the time run forward again, see~\labelcref{eq:Obs2}. At least in the special case~\labelcref{eqMatLawFinSumOfBdOp},
we have the same goal.
\begin{definition}
For a Hilbert space $\HH$ and $\nu\in\R$, we call the unitary mapping
 \begin{equation*}
\TT_\nu\colon
\begin{cases}
\hfill \Lp[\nu]{2} (\R,\HH) &\to\quad \Lp[-\nu]{2} (\R,\HH)\\
\hfill f&\mapsto\quad f(-\cdot)
\end{cases}
\end{equation*}
 the \emph{time-reversal operator}.
\end{definition}
\begin{remark}
Clearly, we have $\TT_{\nu}^{-1}=\TT_{\nu}^{\ast}=\TT_{-\nu}$ for $\nu\in\R$.
\end{remark}
\begin{lemma}\label{lemmaTimeRevTimeDerInterchange}
    For $\nu\in\R$ and a Hilbert space $\HH$, we have $\TT_\nu \partial_{t,\nu} = -\partial_{t,-\nu} \TT_\nu$.
\end{lemma}

\begin{proof}
For $f\in \dom(\partial_{t,\nu})$, we have
\begin{multline*}
\integral{\R}{}{\varphi(s)(\TT_\nu \partial_{t,\nu}f)(s)}{s}=\integral{\R}{}{\varphi(s)(\partial_{t,\nu}f)(-s)}{s}= -\integral{\R}{}{\varphi(-s)(\partial_{t,\nu}f)(s)}{s}\\
= \integral{\R}{}{\frac{\mathrm{d}}{\mathrm{d}s}\varphi(-s)f(s)}{s}
= -\integral{\R}{}{\varphi^\prime(-s)f(s)}{s}=\integral{\R}{}{\varphi^\prime(s)f(-s)}{s}=\integral{\R}{}{\varphi^\prime(s)(\TT_\nu f)(s)}{s}
\end{multline*}
for any $\varphi\in\Cc (\R)$ which implies $\TT_\nu f\in\dom(\partial_{t,-\nu})$ with $\partial_{t,-\nu} \TT_\nu f=-\TT_\nu \partial_{t,\nu}f$. In addition, the same
argument applied to $\TT_{-\nu}=(\TT_{\nu})^{-1}$ shows that $(\TT_\nu)^{-1}(\dom(\partial_{t,-\nu}))= \dom(\partial_{t,\nu})$.
\end{proof}

\begin{lemma}\label{lemmaTimeReversalOfMatOpFinSum}
    Assume we have the material law from~\labelcref{eqMatLawFinSumOfBdOp}, i.e., $M(z)\coloneqq\sum_{k=0}^n z^{-k}M_k$
for $n\in\N\uplus\{0\}$, $z\in\C\setminus\{0\}$, a Hilbert space $\HH$ and $M_0,M_1,\dots,M_n\in L(\HH)$.
 Then, $(\partial_{t,\nu}M(\partial_{t,\nu}))^{\ast_\nu}  = \TT_{\nu}(\partial_{t,\nu}\sum_{k=0}^n \partial_{t,\nu}^{-k}M^\ast_k)\TT_{-\nu}$ for $\nu>0$.
\end{lemma}

\begin{proof}
    This is a direct consequence of \Cref{lemmaAdjTimeDerivMatLaw}, \Cref{lemmaTimeRevTimeDerInterchange} and~\labelcref{eqBITAdjOfFinSum}.
\end{proof}

We want to stress that \Cref{lemmaTimeReversalOfMatOpFinSum} heavily relies on the material law having the form~\labelcref{eqMatLawFinSumOfBdOp}.
We cannot expect a similar statement to hold in the general case. 

\begin{lemma}\label{lemmaTimeReversalOfTensorExtOp}
Consider a Hilbert spaces $\HH$. Let $A\colon\dom(A)\subseteq \HH\to \HH$ be densely defined and closed.
Then for $\nu\in\R$, $\TT_{\nu}A_{\Lp[\nu]{2}}=A_{\Lp[-\nu]{2}}\TT_{\nu}$.
\end{lemma}

\begin{proof}
    This immediately follows from~\labelcref{eqDefinitionTensorTimeLiftingSpatialOperators}.
\end{proof}

\begin{theorem}
     Consider a Hilbert space $\HH$, $\nu>0$, and $M(z)\coloneqq\sum_{k=0}^n z^{-k}M_k$
for $n\in\N\uplus\{0\}$, $z\in\C\setminus\{0\}$ and $M_0,M_1,\dots,M_n\in \Lb(\HH)$, as well as a skew-selfadjoint $A\colon\dom(A)\subseteq \HH\to \HH$.
Then, we have
\begin{equation*}
\TT_{\nu}\Big(\overline{\partial_{t,\nu}\sum_{k=0}^n\nolimits \partial_{t,\nu}^{-k}M^\ast_k-A}\Big)\TT_{-\nu}= (\partial_{t,\nu}M(\partial_{t,\nu})+A)^{\ast_\nu}\text{.}
\end{equation*}
\end{theorem}

\begin{proof}
    This is a combination of \Cref{lemmaTimeReversalOfMatOpFinSum}, \Cref{lemmaTimeReversalOfTensorExtOp} and \Cref{thmBITNuAdjSys}.
\end{proof}
 In other words, in the special case~\labelcref{eqMatLawFinSumOfBdOp}, the dual system~\labelcref{eq:Obs2} should correspond to an evolutionary equation of type
\begin{equation}\label{eqTimeReversedNuAdjointSystemSpCase}
\Big(\partial_{t,\nu}\sum_{k=0}^n \partial_{t,\nu}^{-k}M^\ast_k-A\Big)u=f
\end{equation}
for suitable $u,f\in \Lp[\nu]{2}(\R,\HH)$ and $\nu>0$ (cf.\ \Cref{thmPicard}).
\begin{example}
Applying~\labelcref{eqTimeReversedNuAdjointSystemSpCase} to \Cref{exEvoEq} and \Cref{exNuAdjOfEvoEqEx}, we can provide a few examples:
\begin{itemize}
\item \emph{Heat Equation}\newline
The time-reversed differential operator of the $\nu$-adjoint heat equation reads
\begin{equation*}
\partial_{t,\nu}\begin{pmatrix} 1&0\\ 0&0\end{pmatrix}+\begin{pmatrix} 0&0\\ 0&a^{-\ast}\end{pmatrix}
-\begin{pmatrix} 0&\diverz \\ \grad &0\end{pmatrix}\text{.}
\end{equation*}
\item \emph{Wave Equation}\newline
The time-reversed differential operator of the $\nu$-adjoint wave equation reads
\begin{equation*}
\partial_{t,\nu}\begin{pmatrix} 1&0\\ 0&T^{-1}\end{pmatrix}
+\begin{pmatrix} 0&\diver \\ \gradz &0\end{pmatrix}\text{.}
\end{equation*}
\item \emph{Maxwell's Equations}\newline
The time-reversed differential operator of the $\nu$-adjoint Maxwell's equations read
\begin{equation*}
\partial_{t,\nu}\begin{pmatrix} \varepsilon&0\\ 0&\mu\end{pmatrix}+\begin{pmatrix} \sigma^\ast&0\\ 0&0\end{pmatrix}
-\begin{pmatrix} 0&-\curl \\ \curlz &0\end{pmatrix}\text{.}
\end{equation*}
\end{itemize}
It is easy to see that these time-reversed operators are consistent with the corresponding ones obtained via \eqref{eq:Obs2}.
\end{example}
\subsection{Solution Theory for \texorpdfstring{$\nu$}{nu}-Adjoint Systems}
Apparently, we will have to tweak the existing solution theory for evolutionary equations in order to apply it to the backwards in time
adjoint systems. Analogously to \Cref{thmPicard}, we have the following theorem.
\begin{theorem}\label{thmMainWellPosednessNuAdj}
 Consider a Hilbert space $\HH$, a material law $M\colon\dom(M)\subseteq \C\to \Lb (\HH)$ with
\begin{equation*}
\forall h\in \HH\forall z\in \C,\Real z\geq \nu_0:\Real\iprod[\HH]{h}{zM(z)h}\geq c\norm[\HH]{h}^2 
\end{equation*}
for a $c>0$ and a $\nu_0>\max(\abscb (M),0)$, and a skew-selfadjoint $A\colon\dom(A)\subseteq \HH\to \HH$. Then,
\begin{itemize}
\item the closed subset $\overline{-\partial_{t,-\nu}\LL^\ast_{-\nu}M^\ast(\iu\multm+\nu)\LL_{-\nu}-A}$ of
$\Lp[-\nu]{2} (\R,\HH)\times \Lp[-\nu]{2} (\R,\HH)$ is an (unbounded) operator for all $\nu\geq\nu_0$.
\item For $\nu\geq\nu_0$, $S^{\ast_\nu}_\nu=(\overline{-\partial_{t,-\nu}\LL^\ast_{-\nu}M^\ast(\iu\multm+\nu)\LL_{-\nu}-A})^{-1}\in\Lb (\Lp[-\nu]{2} (\R,\HH))$ with $\norm{S^{\ast_\nu}_\nu}\leq 1/c$,
where $S_{\nu}$ is the solution operator defined in \Cref{thmPicard}.
\item $f\in \Lp[-\nu]{2} (\R,\HH)$ for a $\nu\geq\nu_0$ with $\spt f\subseteq (-\infty,a]$ for an $a\in\R$ implies $\spt S^{\ast_\nu}_\nu f\subseteq (-\infty,a]$.
\item $S^{\ast_\nu}_\nu$ is eventually independent of $\nu$.
\item $f\in\dom(\partial_{t,-\nu})$ for a $\nu\geq\nu_0$ implies $S^{\ast_\nu}_\nu f\in\dom(\partial_{t,-\nu})\cap\dom(A)$, i.e., by~\labelcref{eqSubsetDomainEvEqNuAdj} we get a solution in the sense
of~\labelcref{definNuAdjSysclassEq}.
\end{itemize}
 \end{theorem}
\begin{proof}
 As a first immediate consequence of the properties of the $\nu$-adjoint and~\Cref{thmBITNuAdjSys},
we infer that $\overline{-\partial_{t,-\nu}\LL^\ast_{-\nu}M^\ast(\iu\multm+\nu)\LL_{-\nu}-A}$ is a linear operator with
\begin{equation*}
\overline{-\partial_{t,-\nu}\LL^\ast_{-\nu}M^\ast(\iu\multm+\nu)\LL_{-\nu}-A}^{-1}
=\big(\big(\overline{\partial_{t,\nu}M(\partial_{t,\nu})+A}\big)^{\ast_\nu}\big)^{-1}
=\big(\underbrace{\big(\overline{\partial_{t,\nu}M(\partial_{t,\nu})+A}\big)^{-1}}_{=S_{\nu}}\big)^{\ast_\nu}
\end{equation*}
and $\norm{S^{\ast_\nu}_\nu}=\norm{S_\nu}\leq 1/c$ for all $\nu\geq \nu_0$.

Next, consider $f\in \Lp[-\nu]{2} (\R,\HH)$
with $\spt f\subseteq (-\infty, a]$ for $a\in\R$ and $\nu\geq \nu_0$. For any $g\in \Lp[\nu]{2} (\R,\HH)$, we have
\begin{equation}
\begin{aligned}\label{eqAmnesicDualToCausal1}
  \integral{\R}{}{\iprod[\HH]{g(t)}{\1_{[a,\infty)}(t)(S^{\ast_\nu}_\nu f)(t)}}{t}
  &=\integral{\R}{}{\iprod[\HH]{\1_{[a,\infty)}(t)g(t)}{(S^{\ast_\nu}_\nu f)(t)}}{t}\\
  &=\integral{\R}{}{\iprod[\HH]{S_{\nu}(\1_{[a,\infty)}(t)g(t))}{f(t)}}{t}\text{.}
\end{aligned}
\end{equation}
Causality of $S_{\nu}$ yields $\spt S_{\nu}(\1_{[a,\infty)}g)\subseteq [a,\infty)$. Together with $\spt f\subseteq (-\infty, a]$, this implies
\begin{equation}\label{eqAmnesicDualToCausal2}
    \integral{\R}{}{\iprod[\HH]{g(t)}{\1_{[a,\infty)}(t)(S^{\ast_\nu}_\nu f)(t)}}{t}=0\text{.}
\end{equation}
As $g\in \Lp[\nu]{2} (\R,\HH)$ was arbitrary, the properties of the $\nu$-adjoint show $\1_{[a,\infty)}S^{\ast_\nu}_{\nu} f=0$, i.e., $\spt S^{\ast_\nu}_{\nu} f\subseteq (-\infty, a]$.

For eventual independence, consider $\eta,\nu\geq\nu_0$ and
$f\in \Lp[-\nu]{2} (\R,\HH)\cap \Lp[-\eta]{2} (\R,\HH)$. For a fixed $x\in \HH$ and any $\varphi\in\Cc(\R)$, we have\footnote{Recall that
the eventual independence of the original evolutionary equation and $\varphi(\cdot)\cdot x\in \Lp[\nu]{2} (\R,\HH)\cap \Lp[\eta]{2} (\R,\HH)$ imply
$S_\nu(\varphi(\cdot)\cdot x)=S_\eta(\varphi(\cdot)\cdot x)$.}
\begin{align*}
\integral{\R}{}{\overline{\varphi(t)}\iprod[\HH]{x}{(S^{\ast_\nu}_\nu f)(t)&-(S^{\ast_\eta}_\eta f)(t)}}{t}
=\integral{\R}{}{\iprod[\HH]{\varphi(t)\cdot x}{(S^{\ast_\nu}_\nu f)(t)-(S^{\ast_\eta}_\eta f)(t)}}{t}\\
&=\integral{\R}{}{\iprod[\HH]{\varphi(t)\cdot x}{(S^{\ast_\nu}_\nu f)(t)}}{t}-\integral{\R}{}{\iprod[\HH]{\varphi(t)\cdot x}{(S^{\ast_\eta}_\eta f)(t)}}{t}\\
&=\integral{\R}{}{\iprod[\HH]{S_\nu(\varphi(\cdot)\cdot x)(t)}{f(t)}}{t}-\integral{\R}{}{\iprod[\HH]{S_\eta(\varphi(\cdot)\cdot x)(t)}{f(t)}}{t}\\
&=0\text{.}
\end{align*}
Hence, the fundamental lemma of calculus of variations yields
 $\iprod[\HH]{x}{(S^{\ast_\nu}_\nu f)(t)-(S^{\ast_\eta}_\eta f)(t)}=0$ for $t\in \R$.
A density argument implies
$S^{\ast_\nu}_\nu f=S^{\ast_\eta}_\eta f$.

Finally, consider some $\nu\geq \nu_0$. We know that $S_{\nu}\partial_{t,\nu}\subseteq \partial_{t,\nu}S_{\nu}$ (cf.~\cite[Remark~6.3.4]{SeTrWa22}). Applying the $\nu$-adjoint first and
then \Cref{lemmaSumProdOfNuAdjRel} as well as \Cref{remNuAdjOfTimeDeriv}, we obtain $\overline{-S^{\ast_\nu}_{\nu}\partial_{t,-\nu}}\subseteq -\partial_{t,-\nu}S^{\ast_\nu}_{\nu}$, i.e.,
we get $S^{\ast_\nu}_{\nu}\partial_{t,-\nu}\subseteq \partial_{t,-\nu}S^{\ast_\nu}_{\nu}$. Moreover by $M(\partial_{t,\nu})\partial_{t,\nu}\subseteq \partial_{t,\nu}M(\partial_{t,\nu})$, we have
\begin{align*}
S_{\nu}A\restriction_{\dom(\partial_{t,\nu})}&=1-S_{\nu}\partial_{t,\nu}M(\partial_{t,\nu})\restriction_{\dom(\partial_{t,\nu})\cap\dom(A)}
\intertext{and once again by $S_{\nu}\partial_{t,\nu}\subseteq \partial_{t,\nu}S_{\nu}$ that}
&=1-\partial_{t,\nu}S_{\nu}M(\partial_{t,\nu})\restriction_{\dom(\partial_{t,\nu})\cap\dom(A)}\\
&\subseteq 1-\partial_{t,\nu}S_{\nu}M(\partial_{t,\nu})\text{.}
\end{align*}
Applying the $\nu$-adjoint first and
then \Cref{lemmaSumProdOfNuAdjRel} as well as \Cref{remNuAdjOfTimeDeriv} and \Cref{remDualMatLawProperties}, we obtain\footnote{Recall $A^{\ast_\nu}=-A$ and
$\overline{A\restriction_{\dom(\partial_{t,\nu})}}=A$.}
 $1+\LL^\ast_{-\nu}M^\ast(\iu\multm+\nu)\LL_{-\nu}S^{\ast_\nu}_{\nu}\partial_{t,-\nu}\subseteq -AS^{\ast_\nu}_{\nu}$.
\end{proof}

\begin{definition}
Following~\cite[Chapter~3.2]{Trostorff2018}, we call a linear operator $S\colon\dom(S)\subseteq \Lp[\nu]{2} (\R,\HH)\to\Lp[\nu]{2} (\R,\HH)$ for a Hilbert space $\HH$ and $\nu \in\R$
\emph{amnesic} if $\spt f\subseteq (-\infty,a]$ implies $\spt Sf\subseteq (-\infty,a]$ for all $f\in\dom(S)$ and $a\in\R$. 
\end{definition}
In \Cref{thmMainWellPosednessNuAdj} we have shown that the solution operator corresponding to a well-posed $\nu$-adjoint system is amnesic. In fact, we have the following more general statement which can be seen as a (bounded) counterpart of \cite[Lemma~3.2.7]{Trostorff2018}.
\begin{corollary}
For a Hilbert space $\HH$ and $\nu \in\R$, a bounded operator $S\in\Lb(\Lp[\nu]{2} (\R,\HH))$ is amnesic if and only if $S^{\ast_\nu}\in\Lb(\Lp[-\nu]{2} (\R,\HH))$ is
causal.
\end{corollary}
\begin{proof}
Arguing similarly to~\labelcref{eqAmnesicDualToCausal1} and~\labelcref{eqAmnesicDualToCausal2} yields the statement.
\end{proof}
\section{Control Theory for Evolutionary Equations}
\label{sec:ControlTheory}

In this section, we focus on control theory for evolutionary equations.

%
%
%
%
%
%

\subsection{Duality Statements}

We start with a version of Douglas' lemma. For a proof, see, e.g., \cite[Theorem 1]{Douglas-66} and \cite[Theorem 2.2]{Carja1988}.
For a normed space $X$ and $c\geq 0$, we write $K_c\coloneqq \{x\in X \mid \norm[X]{x}\leq c\}$.
\begin{lemma}
\label{lem:Douglas}
    Let $\HH_0, \HH_1, \HH$ be Hilbert spaces,  $A\in\Lb(\HH_1,\HH)$ and $B\in \Lb(\HH_0,\HH)$. The following are equivalent:
    \begin{enumerate}[label=\normalfont{(\roman*)}]
     \item $\ran(A) \subseteq \ran(B)$.
     \item There exists $C\in \Lb(\HH_1,\HH_0)$ such that $A=BC$.
     \item There exists $c\geq 0$ such that $A(K_1) \subseteq B(K_c)$.
     \item There exists $c\geq 0$ such that $\norm[\HH_1]{A^*x} \leq c \norm[\HH_0]{B^*x}$ for all $x\in \HH$.
    \end{enumerate}    
\end{lemma}
\Cref{lem:Douglas} remains true for Banach spaces instead of Hilbert spaces as long as $\HH_0$ is reflexive.
\begin{lemma}
\label{lem:range_dual}
    Let $\nu\in\R$, $\HH_0,\HH$ be Hilbert spaces,  $A\in\Lb\bigl(\Lp[\nu]{2}(\R,\HH)\bigr)$ and $B\in \Lb\bigl(\Lp[\nu]{2} (\R,\HH_0), \Lp[\nu]{2} (\R,\HH)\bigr)$. The following are equivalent:
    \begin{enumerate}[label=\normalfont{(\roman*)}]
     \item $\ran(A) \subseteq \ran(B)$.
     \item There exists $c\geq 0$ such that  $\norm{A^{\ast_\nu}g}_{\Lp[-\nu]{2}} \leq c \norm{B^{\ast_\nu}g}_{\Lp[-\nu]{2}}$ for all $g\in \Lp[-\nu]{2} (\R,\HH)$.
    \end{enumerate}    
\end{lemma}

\begin{proof}
    (i)$\Rightarrow$(ii):
    By Lemma \ref{lem:Douglas}, there exists $c\geq 0$ such that $A(K_1) \subseteq B(K_c)$. Let $g\in \Lp[-\nu]{2}(\R,\HH)$. Then,
    \begin{align*}
        \norm{A^{\ast_\nu}g}_{\Lp[-\nu]{2}} & = \sup_{f\in \Lp[\nu]{2}(\R,\HH),\,\norm{f}_{\Lp[\nu]{2}}\leq 1} \abs{\dprod{f}{A^{\ast_\nu}g}_{\nu}} \\
        & = \sup_{f\in \Lp[\nu]{2}(\R,\HH),\,\norm{f}_{\Lp[\nu]{2}}\leq 1} \abs{\dprod{Af}{g}_{\nu}} \\
        & \leq \sup_{f\in \Lp[\nu]{2}(\R,\HH_0),\,\norm{f}_{\Lp[\nu]{2}}\leq c} \abs{\dprod{Bf}{g}_{\nu}} \\
        & = \sup_{f\in \Lp[\nu]{2}(\R,\HH_0),\,\norm{f}_{\Lp[\nu]{2}}\leq c} \abs{\dprod{f}{B^{\ast_\nu}g}_{\HH_0,\nu}} \\
        & = c \norm{B^{\ast_\nu}g}_{\Lp[-\nu]{2}}. 
    \end{align*}
    
    (ii)$\Rightarrow$(i):
    Note that $M\coloneqq A(K_1)$ and $N\coloneqq B(K_c)$ are convex, and $N$ is closed by the Banach--Alaoglu theorem. 
    For all $g\in \Lp[-\nu]{2}(\R,\HH)$ we estimate
    \begin{align*}
        \sup_{f\in M} \abs{\dprod{f}{g}_{\nu}} = \norm{A^{\ast_\nu}g}_{\Lp[-\nu]{2}} \leq c \norm{B^{\ast_\nu}g}_{\Lp[-\nu]{2}} = \sup_{f\in N} \abs{\dprod{f}{g}_{\nu}}.
    \end{align*}
  Since $\exp(2\nu \multm)\from \Lp[-\nu]{2}(\R,\HH)\to \Lp[\nu]{2}(\R,\HH)$ is unitary,~\labelcref{eq:ConnWeiInnProdNuProd} implies
    \[\sup_{f\in M} \abs{\iprod{f}{h}_{\Lp[\nu]{2}}} \leq \sup_{f\in N} \abs{\iprod{f}{h}_{\Lp[\nu]{2}}}\]
    for all $h\in \Lp[\nu]{2}(\R,\HH)$.
    From \cite[Lemma 2.9]{EGST2024} 
    , we obtain $M\subseteq N$ and Lemma \ref{lem:Douglas} yields the assertion. 
\end{proof}

%
%
%

\subsection{Null-Controllability for Evolutionary Equations}\label{sec:NullControlWithoutPwise}

Let $\HH, \HH_0$ be Hilbert spaces, $\nu>0$, $M$ a material law and $A$ skew-selfadjoint in $\HH$. Let $B\in\Lb(\HH_0,\HH)$. Like $A$ (see~\labelcref{eqDefinitionTensorTimeLiftingSpatialOperators}), we lift $B$ via action as a multiplication operator such that $B\in\Lb\bigl(\Lp[\nu]{2}(\R,\HH_0),\Lp[\nu]{2}(\R,\HH)\bigr)$. Furthermore, assume 
\begin{equation*}
\forall h\in \HH\forall z\in \C,\Real z\geq \nu_0:\Real\iprod[\HH]{h}{zM(z)h}\geq c\norm[\HH]{h}^2 
\end{equation*}
for a $c>0$ and a $\nu_0>\max(\abscb (M),0)$.
 Then, we consider the evolutionary equation
\begin{equation}
\label{eq:EE_control}
    \bigl(\partial_t M(\partial_t) + A\bigr) U = F+ BG.
\end{equation}
Let $S_\nu\coloneqq \bigl(\overline{\partial_{t,\nu} M(\partial_{t,\nu}) + A}\bigr)^{-1} \in \Lb\bigl(\Lp[\nu]{2}(\R,\HH)\bigr)$ be the corresponding solution operator from
\Cref{thmPicard}.

\begin{definition}
    Let $T >0$. We call \eqref{eq:EE_control} \emph{null-controllable} in time $T$ if for all $F\in \Lp[\nu]{2}(\R,\HH)$ there exists $G\in \Lp[\nu]{2}(\R,\HH_0)$ such that $\spt S_\nu(F+BG) \subseteq (-\infty,T]$.
\end{definition}

For $T>0$ we identify
\begin{align*}
    \Lp[\nu]{2}((-\infty,T],\HH) & \coloneqq \{f\in \Lp[\nu]{2}(\R,\HH):\; \spt f \subseteq (-\infty,T]\},\\
    \Lp[\nu]{2}([T,\infty),\HH) & \coloneqq \{f\in \Lp[\nu]{2}(\R,\HH):\; \spt f \subseteq [T,\infty)\}.    
\end{align*}
Thus, null-controllability in time $T$ can be rephrased as for all $F\in \Lp[\nu]{2}(\R,\HH)$ there exists $G\in \Lp[\nu]{2}(\R,\HH_0)$ such that $S_\nu(F+BG) \in  \Lp[\nu]{2}((-\infty,T],\HH)$.

For $\nu\in \R$ and $T\in\R$ we define $\1_{\leq T,\nu}\from \Lp[\nu]{2}((-\infty,T],\HH)\to \Lp[\nu]{2}(\R,\HH)$ and $\1_{\geq T,\nu}\from \Lp[\nu]{2}([T,\infty),\HH)\to \Lp[\nu]{2}(\R,\HH)$ to be the canonical embeddings.
We write $r_{\leq T,\nu}\coloneqq \1_{\leq T,\nu}\1_{\leq T,\nu}^{\ast}$ and $r_{\geq T,\nu}\coloneqq \1_{\geq T,\nu}\1_{\geq T,\nu}^{\ast}$ for the restriction maps.

\begin{lemma}\label{lemma:AdjointCanonEmbIndFct}
 Let $\nu\in\R$ and $T >0$. Then $\1_{\leq T,\nu}^{\ast}\from \Lp[\nu]{2}(\R,\HH) \to \Lp[\nu]{2}((-\infty,T],\HH)$ is given by 
 \[\1_{\leq T,\nu}^{\ast} f = f_T\]
for $f\in  \Lp[\nu]{2}(\R,\HH)$, where $f_T\in \Lp[\nu]{2}((-\infty,T],\HH)$ is defined to coincide with $f$ on $(-\infty,T]$. Furthermore, $(r_{\leq T,\nu})^{\ast_\nu}=r_{\leq T,-\nu}$.

Analogous statements for $\1_{\geq T,\nu}^{\ast}$ and $r_{\geq T,\nu}$ hold.
\end{lemma}
\begin{proof}
    Let $f\in \Lp[\nu]{2}(\R,\HH)$ and $g\in \Lp[\nu]{2}((-\infty,T],\HH)$. Then clearly,
    \[\iprod{g}{\1_{\leq T,\nu}^{\ast}f}_{\Lp[\nu]{2}} =\iprod{\1_{\leq T,\nu}g}{f}_{\Lp[\nu]{2}} = \iprod{\1_{\leq T,\nu}g}{\1_{\leq T,\nu}f_T}_{\Lp[\nu]{2}}=
    \iprod{g}{f_T}_{\Lp[\nu]{2}}.\]    
    Next, for $f\in \Lp[\nu]{2}(\R,\HH)$ and $g\in \Lp[-\nu]{2}(\R,\HH)$, 
    \begin{multline*}
    \dprod{r_{\leq T,\nu}f}{g}_{\nu} = \int_\R \iprod[\HH]{\1_{\leq T,\nu}f_T(t)}{g(t)}\,dt = \int_{-\infty}^T \iprod[\HH]{f(t)}{g(t)}\,dt\\ =\int_\R \iprod[\HH]{f(t)}{\1_{\leq T,-\nu}g_T(t)}\,dt = \dprod{f}{r_{\leq T,-\nu}g}_{\nu}.\end{multline*}
    \end{proof}

\begin{theorem}
\label{thm:control_equivalence}
    Let $T>0$.
    The following are equivalent:
    \begin{enumerate}[label=\normalfont{(\roman*)}]
     \item \eqref{eq:EE_control} is null-controllable in time $T$.
     \item $\ran(r_{\geq T,\nu} S_\nu) \subseteq \ran(r_{\geq T,\nu} S_\nu B)$.
     \item There exists $c\geq 0$ with $\norm{(r_{\geq T,\nu} S_\nu)^{\ast_\nu} F}_{\Lp[-\nu]{2}} \leq c \norm{(r_{\geq T,\nu} S_\nu B)^{\ast_\nu} F}_{\Lp[-\nu]{2}}$ for all $F\in \Lp[-\nu]{2}(\R,\HH)$.
     \item There exists $c\geq 0$with $\norm{S_\nu^{\ast_\nu} \1_{\geq T,-\nu} F}_{\Lp[-\nu]{2}} \leq c \norm{B^{\ast_\nu} S_\nu^{\ast_\nu} \1_{\geq T,-\nu} F}_{\Lp[-\nu]{2}}$ for all $F\in \Lp[-\nu]{2}([T,\infty),\HH)$.
    \end{enumerate}
\end{theorem}

\begin{proof}
    (i)$\Rightarrow$(ii): 
    Let $h\in \ran(r_{\geq T,\nu}S_\nu)$, i.e., there exists $f\in \Lp[\nu]{2}(\R,\HH)$ such that $r_{\geq T,\nu} S_\nu f = h$. By (i), there exists $g\in \Lp[\nu]{2}(\R,\HH_0)$ such that $\spt S_\nu(f+Bg) \subseteq(-\infty,T]$, i.e.\ $r_{\geq T,\nu} S_\nu (f+Bg) = 0$. This yields $h = r_{\geq T,\nu} S_\nu f = r_{\geq T,\nu} S_\nu B(-g)$; that is, $h\in \ran(r_{\geq T,\nu} S_\nu B)$.
    
    (ii)$\Rightarrow$(i): 
    Let $f\in \Lp[\nu]{2}(\R,\HH)$ and $h\coloneqq r_{\geq T,\nu} S_\nu f\in \ran(r_{\geq T,\nu} S_\nu)$. By (ii) there exists $g\in \Lp[\nu]{2}(\R,\HH_0)$ such that $h = r_{\geq T,\nu} S_\nu Bg$. Then, $r_{\geq T,\nu} S_\nu(f+B(-g)) = 0$; that is, $\spt S_\nu(f+B(-g))\subseteq (-\infty,T]$, so \eqref{eq:EE_control} is null-controllable in time $T$.
    
    (ii)$\Leftrightarrow$(iii):
    This is a consequence of Lemma \ref{lem:range_dual}.
    
    (iii)$\Leftrightarrow$(iv):
    The claim follows from
    \Cref{lemma:AdjointCanonEmbIndFct} and $(r_{\geq T,\nu} S_\nu)^{\ast_\nu} = S_\nu^{\ast_\nu} r_{\geq T,-\nu}$ and $(r_{\geq T,\nu} S_\nu B)^{\ast_\nu} = B^{\ast_\nu} S_\nu^{\ast_\nu} r_{\geq T,-\nu}$.
\end{proof}

\begin{example}
    We get back to our set of examples.
    \begin{itemize}
\item \emph{Heat Equation}: Let $\Omega\subseteq \R^d$ be open, $\Omega_0\subseteq \Omega$, and let $\HH\coloneqq \Lp{2}(\Omega)\times \Lp{2}(\Omega)^d$ and $\HH_0\coloneqq \Lp{2}(\Omega_0)$. Let $B_0\from \Lp{2}(\Omega_0)\to \Lp{2}(\Omega)$ be the embedding by extending functions on $\Omega_0$ to $\Omega$ by zero, and let $B\in \Lb(\HH_0,\HH)$ be given by $B\coloneqq \begin{pmatrix}B_0 \\ 0 \end{pmatrix}$. Then the controlled Heat Equation is given by
\begin{equation*}
\bigg(\partial_{t,\nu}\begin{pmatrix} 1&0\\ 0&0\end{pmatrix}+\begin{pmatrix} 0&0\\ 0&a^{-1}\end{pmatrix} 
+\begin{pmatrix} 0&\diverz \\ \grad &0\end{pmatrix}\bigg)\begin{pmatrix} \vartheta\\ q\end{pmatrix}
=\begin{pmatrix} Q\\ 0\end{pmatrix} + BG \text{.}
\end{equation*}
For $T>0$, the Heat Equation is null-controllable in time $T$ if and only if there exists $c\geq 0$ such that for all $F\in \Lp[-\nu]{2}([T,\infty),\HH)$ and the solution $\begin{pmatrix} \vartheta\\q\end{pmatrix}$ of the $\nu$-adjoint system
\begin{equation*}
\left(-\partial_{t,-\nu}\begin{pmatrix} 1&0\\ 0&0\end{pmatrix}+\begin{pmatrix} 0&0\\ 0&a^{-\ast}\end{pmatrix}
-\begin{pmatrix} 0&\diverz \\ \grad &0\end{pmatrix}\right)\begin{pmatrix}\vartheta\\q\end{pmatrix} = \1_{\geq T} F
\end{equation*}
satisfies
$\norm{\begin{pmatrix} \vartheta \\q\end{pmatrix}}_{\Lp[-\nu]{2}} \leq c \norm{\vartheta\restriction_{\Omega_0}}_{\Lp[-\nu]{2}}$. 

\item \emph{Wave Equation}
Let $\Omega\subseteq \R^d$ be open, $\Omega_0\subseteq \Omega$, and let $\HH\coloneqq \Lp{2}(\Omega)\times \Lp{2}(\Omega)^d$ and $\HH_0\coloneqq \Lp{2}(\Omega_0)$. Let $B_0\from \Lp{2}(\Omega_0)\to \Lp{2}(\Omega)$ be the embedding by extending functions on $\Omega_0$ to $\Omega$ by zero, and let $B\in \Lb(\HH_0,\HH)$ be given by $B\coloneqq \begin{pmatrix}B_0 \\ 0 \end{pmatrix}$. Then the controlled Wave Equation is given by
\begin{equation*}
\bigg(\partial_{t,\nu}\begin{pmatrix} 1&0\\ 0&T^{-1}\end{pmatrix}
-\begin{pmatrix} 0&\diver \\ \gradz &0\end{pmatrix}\bigg)\begin{pmatrix} v\\ \sigma\end{pmatrix}
=\begin{pmatrix} f\\ 0\end{pmatrix} + BG\text{.}
\end{equation*}
For $S>0$, the Wave Equation is null-controllable in time $S$ if and only if there exists $c\geq 0$ such that for all $F\in \Lp[-\nu]{2}([T,\infty),\HH)$ and the solution $\begin{pmatrix} v\\\sigma\end{pmatrix}$ of the $\nu$-adjoint system
\begin{equation*}
\left(-\partial_{t,-\nu}\begin{pmatrix} 1&0\\ 0&T^{-1}\end{pmatrix}
+\begin{pmatrix} 0&\diverz \\ \grad &0\end{pmatrix}\right)\begin{pmatrix}v\\\sigma\end{pmatrix} = \1_{\geq S} F
\end{equation*}
satisfies
$\norm{\begin{pmatrix} v \\\sigma\end{pmatrix}}_{\Lp[-\nu]{2}} \leq c \norm{v\restriction_{\Omega_0}}_{\Lp[-\nu]{2}}$.

\item 
Let $\Omega\subseteq \R^d$ be open, $\Omega_0\subseteq \Omega$, and let $\HH\coloneqq \Lp{2}(\Omega)^3\times \Lp{2}(\Omega)^3$, and $\HH_0\coloneqq \Lp{2}(\Omega_0)^3\times \Lp{2}(\Omega_0)^3$. Let $B\from \HH_0\to \HH$ be the embedding by extending functions on $\Omega_0$ to $\Omega$ by zero. Then the controlled Maxwell's Equations are given by
\begin{equation*}
\bigg(\partial_{t,\nu}\begin{pmatrix} \varepsilon&0\\ 0&\mu\end{pmatrix}+\begin{pmatrix} \sigma&0\\ 0&0\end{pmatrix}
+\begin{pmatrix} 0&-\curl \\ \curlz &0\end{pmatrix}\bigg)\begin{pmatrix} E\\ H\end{pmatrix}
=\begin{pmatrix} j_0\\ 0\end{pmatrix} + BG\text{.}
\end{equation*}
For $t\in \R$, the Maxwell's Equations are null-controllable in time $T$ if and only if there exists $c\geq 0$ such that for all $F\in \Lp[-\nu]{2}([T,\infty),\HH)$ and the solution $\begin{pmatrix} E\\H\end{pmatrix}$ of the $\nu$-adjoint system
\begin{equation*}
\left(-\partial_{t,-\nu}\begin{pmatrix} \varepsilon&0\\ 0&\mu\end{pmatrix}+\begin{pmatrix} \sigma^\ast&0\\ 0&0\end{pmatrix}
-\begin{pmatrix} 0&-\curl \\ \curlz &0\end{pmatrix}\right)\begin{pmatrix} E\\H\end{pmatrix} =  \1_{\geq T} F
\end{equation*}
satisfies
$\bignorm{\begin{pmatrix} E \\H\end{pmatrix}}_{\Lp[-\nu]{2}} \leq c \bignorm{\begin{pmatrix} E \\H\end{pmatrix}\restriction_{\Omega_0}}_{\Lp[-\nu]{2}}$. \qedhere
\end{itemize}
 \end{example}

\subsection{Pointwise Statements for Null-Controllability}
\label{subsec:InitialValues}

In this subsection, we specialise to a particular class of material laws that covers all the examples we treat in this paper. To this end, we consider the same assumptions as in
\Cref{sec:NullControlWithoutPwise}, and we additionally ask the material law to be
 of the form $M(z)\coloneqq M_0 + z^{-1} M_1$, where $M_0,M_1\in \Lb(\HH)$ and $\Real z >0$. In this particular situation, we can obtain a pointwise interpretation of solutions of evolutionary equations. Note that these material laws describe exactly those equations without memory terms \cite[Proposition 3.2.10]{Trostorff2018}. We start with a version of \cite[Theorem 9.4.3]{SeTrWa22} which includes inhomogeneities at the right-hand side.
We write 
\begin{align*}
    \sH^1_\nu(\R,\HH) & \coloneqq\dom(\partial_{t,\nu}),\\
    \sH^{-1}_\nu(\R,\HH) & \coloneqq \dom(\partial_{t,\nu})',
\end{align*}
as well as
\[\sH^{-1}(A) \coloneqq \dom(A)';\]
cf.\ \cite[Section 9.2]{SeTrWa22}.

The next theorem is a slight generalisation of \cite[Theorem 9.4.3]{SeTrWa22}, where only $F=0$ was treated.
\begin{theorem}[{compare \cite[Theorem 9.4.3]{SeTrWa22}}]
\label{thm:Solution_InitialValues}
    Let $U_0\in \dom(A)\subseteq\HH$, $U\in \Lp[\nu]{2}(\R,\HH)$, $F\in \Lp[\nu]{2}([0,\infty),\HH)$.
    Then, the following are equivalent:
    \begin{enumerate}[label=\normalfont{(\roman*)}]
     \item
        $M_0U - \1_{[0,\infty)}M_0U_0\from \R\to \sH^{-1}(A)$ is continuous, $\spt U\subseteq [0,\infty)$ and
        \begin{align*}
            \partial_{t,\nu} M_0U + M_1 U + AU & = F \quad\text{on}\; (0,\infty),\\
            M_0 U(0+) & = M_0U_0  \quad\text{in}\; \sH^{-1}(A).
        \end{align*}
     \item
        $U-\1_{[0,\infty)}U_0 \in \dom\bigl(\overline{\partial_{t,\nu}M_0+M_1+A}\bigr)$ and \[\bigl(\overline{\partial_{t,\nu}M_0+M_1+A}\bigr)(U-\1_{[0,\infty)}U_0) = F - \1_{[0,\infty)}(M_1+A)U_0\in\Lp[\nu]{2}([0,\infty),\HH)\text{.} \]
     \item
        $U = S_\nu(F + \delta_0M_0U_0)$, where $\delta_0$ is the Dirac distribution centered at $0$ and we extended $S_\nu\in \Lb(\sH_{\nu}^{-1}(\R,\HH))$.
    \end{enumerate}
    In either case, $M_0U - \1_{[0,\infty)}M_0U_0 \in \sH_\nu^1(\R,\sH^{-1}(A))$.
\end{theorem}

\begin{proof}
    The proof is analogous to the one in \cite[Theorem 9.4.3]{SeTrWa22}; one only has to take into account the additional term $F$ on the right-hand side.
\end{proof}

Let us now turn to null-controllability for evolution equations with material laws $M$ of the form $M(z)\coloneqq M_0 + z^{-1} M_1$.

\begin{definition}
    We say that
    \begin{equation}\label{eq:PWiseNullControlSystem}
    \begin{aligned}
            \partial_{t,\nu} M_0U + M_1 U + AU & = BG \quad\text{on}\; (0,\infty),\\
            M_0 U(0+) & = M_0U_0  \quad\text{in}\; \sH^{-1}(A),
    \end{aligned}
    \end{equation}
    is \emph{pointwise null-controllable} in time $T>0$ if for all 
    $U_0\in \dom(A)$
    there exists $G\in \Lp[\nu]{2}(\R,\HH_0)$ with $\spt G\subseteq [0,\infty)$ such that    
    the solution $U\in \Lp[\nu]{2}(\R,\HH)$ given by
    \[U = S_\nu\Bigl(BG-\1_{[0,\infty)}(M_1+A)U_0\Bigr) + \1_{[0,\infty)}U_0\]
    satisfies $H^{-1}(A)\ni M_0U(T)=0$.
\end{definition}
Note that in view of Theorem \ref{thm:Solution_InitialValues} we have $M_0U -\1_{[0,\infty)}M_0U_0 \in H^1_\nu(\R,H^{-1}(A))$ so we can actually perform the point evaluation at $T$.

For $U_0\in\dom(A)$ and $G\in \Lp[\nu]{2}(\R,\HH_0)$ with $\spt G\subseteq [0,\infty)$ let $U^G\in \Lp[\nu]{2}(\R,\HH)$ be given by
\[U^G \coloneqq S_\nu\Bigl(BG-\1_{[0,\infty)}(M_1+A)U_0\Bigr) + \1_{[0,\infty)}U_0.\]
Moreover, let $U\in \Lp[\nu]{2}(\R,\HH)$ be given by
\[U \coloneqq S_\nu\Bigl(-\1_{[0,\infty)}(M_1+A)U_0\Bigr) + \1_{[0,\infty)}U_0.\]
Then $M_0U^G -\1_{[0,\infty)}M_0U_0, M_0U -\1_{[0,\infty)}M_0U_0 \in \sH^1_\nu(\R,\sH^{-1}(A))$ and therefore also their difference
\[M_0S_\nu(BG) = M_0(U^G-U) \in \sH^1_\nu(\R,\sH^{-1}(A)).\]


\begin{theorem}
\label{thm:control_equivalence_InitialValues}
    Let $T>0$. The following are equivalent:
    \begin{enumerate}[label=\normalfont{(\roman*)}]
     \item \labelcref{eq:PWiseNullControlSystem} is pointwise null-controllability.
     \item $\ran\Bigl(M_0 S_\nu\bigl(\1_{[0,\infty)}(M_1+A)(\cdot)\bigr)(T) - M_0(\cdot)\Bigr) \subseteq \ran\Bigl(M_0 S_\nu B\1_{\geq T,\nu} (\cdot)(T)\Bigr)$.
    \end{enumerate} 
\end{theorem}

\begin{proof}
    (i)$\Rightarrow$(ii): 
    Let $h\in \ran\Bigl(M_0 S_\nu\bigl(\1_{[0,\infty)}(M_1+A)(\cdot)\bigr)(T) - M_0(\cdot)\Bigr)$. Then, there exists $U_0 \in \dom(A)$ such that $h =
    M_0 S_\nu\bigl(\1_{[0,\infty)}(M_1+A)U_0\bigr)(T) - M_0U_0$. By (i), there exists $G\in \Lp[\nu]{2}(\R,\HH_0)$ with $\spt G\subseteq [0,\infty)$ such that $M_0S_\nu\bigl(BG-\1_{[0,\infty)}(M_1+A)U_0\bigr)(T) + M_0U_0 = 0$. Thus,
    \[h = M_0S_\nu \bigl(BG\bigr)(T) \in \ran\Bigl(M_0S_\nu \bigl(B\1_{\geq T,\nu} (\cdot)\bigr)(T)\Bigr).\]

    (ii)$\Rightarrow$(i):
    Let $U_0\in\dom(A)$. Then by (ii) there exists some $G\in \Lp[\nu]{2}(\R,\HH_0)$ with $\spt G\subseteq [0,\infty)$ such that $M_0 S_\nu\bigl(\1_{[0,\infty)}(M_1+A)U_0\bigr)(T) - M_0U_0 = M_0S_\nu \bigl(BG\bigr)(T)$.
    
    Let $U\coloneqq S_\nu\Bigl(BG-\1_{[0,\infty)}(M_1+A)U_0\Bigr) + \1_{[0,\infty)}U_0 \in \Lp[\nu]{2}(\R,\HH)$.
    Then, $U-\1_{[0,\infty)}U_0\from \R\to H^{-1}(A)$ is continuous, and thus
    \[M_0U(T) = M_0 S_\nu\Bigl(BG-\1_{[0,\infty)}(M_1+A)U_0\Bigr)(T) + M_0U_0 = 0. \qedhere\]
%
\end{proof}

 In view of Theorem \ref{thm:control_equivalence_InitialValues}, we are missing is a corresponding observability inequality for the $\nu$-adjoint operators. The key difficulty here is to cope with the point evaluations with values in $\sH^{-1}(A)$ and their representations in $\Lp[\nu]{2}(\R,\sH^{-1}(A))$. We thus formulate this as an open problem.
 
\begin{openproblem}
   Find a suitable observability inequality for the $\nu$-adjoint operators in the context of Theorem \ref{thm:control_equivalence_InitialValues}.
\end{openproblem} 



\end{document}